\documentclass[11  pt, oneside]{amsart}   
\usepackage[margin=1.25in,lmargin=.75in,rmargin=1.75in,bmargin=1in,tmargin=1in]{geometry}      

\usepackage{graphicx}
\usepackage{booktabs} 
\usepackage{listings}
\usepackage{amsmath}
\usepackage{algorithm}
\usepackage[noend]{algpseudocode}
\usepackage{algorithmicx}
\usepackage{enumitem}
\usepackage{url}
\usepackage{microtype}

\usepackage[foot]{amsaddr}

\usepackage{simple-math}

\usepackage{amsthm}
\newtheorem{theorem}{Theorem}
\newtheorem{corollary}[theorem]{Corollary}
\newtheorem{lemma}[theorem]{Lemma}

\usepackage[T1]{fontenc}
\usepackage[utf8]{inputenc}

\begin{document}
	
\title{A Projection Method for Metric-Constrained Optimization}
\author{Nate Veldt$^1$}
\address{$^1$Purdue University, Mathematics Department}
\author{David F. Gleich$^2$} 
\address{$^2$Purdue University, Computer Science Department}
\author{Anthony Wirth$^3$}
\address{$^3$The University of Melbourne, Computing and Information Systems School}
\author{James Saunderson$^4$}
\address{$^4$Monash University, Department of Electrical and Computer Systems Engineering}

\begin{abstract}
	We outline a new approach for solving optimization problems which enforce triangle inequalities on output variables. We refer to this as metric-constrained optimization, and give several examples where problems of this form  arise in machine learning applications and theoretical approximation algorithms for graph clustering. Although these problem are interesting from a theoretical perspective, they are challenging to solve in practice due to the high memory requirement of black-box solvers. In order to address this challenge we first prove that the metric-constrained linear program relaxation of correlation clustering is equivalent to a special case of the metric nearness problem. We then developed a general solver for metric-constrained linear and quadratic programs by generalizing and improving a simple projection algorithm originally developed for metric nearness. We give several novel approximation guarantees for using our framework to find lower bounds for optimal solutions to several challenging graph clustering problems. We also demonstrate the power of our framework by solving optimizing problems involving up to $10^8$ variables and $10^{11}$ constraints.
\end{abstract}

\maketitle

\section{Introduction}
Learning pairwise distance scores among objects in a dataset is an important task in machine learning and data mining. Principled methods abound for how to approach this task in different contexts. A number of these methods rely on solving a convex optimization problem involving metric constraints of the form $x_{ij} \leq x_{ik} + x_{jk}$, where, for instance, $x_{ij}$ represents the distance score between objects $i$ and $j$. Solving metric-constrained optimization problems has been applied to semi-supervised clustering~\cite{biswas2014semi,batra2008semi}, joint clustering of image segmentations~\cite{Vitaladevuni2010coclustering,glasner2011contour}, the metric nearness problem~\cite{brickell2008metricnearness,dhillon2003MNreport,dhillon2004tfa}, and sensor location~\cite{gentile2005sensor,gentile2007distributed}. Metric-constrained linear programs also arise frequently as convex relaxations for NP-hard graph clustering problems~\cite{chawla2015near,veldt2017lamcc,leighton1999multicommodity,Agarwal2008metricmod,poljak1995maximum}.

Obtaining distance scores by solving a metric-constrained optimization problem is of great theoretical interest, but can be very challenging in practice since these problems often involve $O(n^2)$ variables and $O(n^3)$ constraints, where $n$ is the number of items in a dataset or number of nodes in a derived graph. In this paper we specifically consider a class of linear programming relaxations of NP-hard graph clustering objectives that have been extensively studied in theory, but rarely solved in practice due to the memory constraints of traditional black-box optimization software. The goal in our work is to develop {practical} solvers for these and other related metric-constrained optimization tasks.

The starting point in our work is an observation that the linear programming relaxation for correlation clustering~\cite{Bansal2004correlation,chawla2015near} is equivalent to a special case of the $\ell_1$ metric nearness problem~\cite{brickell2008metricnearness}. 
Based on this, we develop a general strategy for metric-constrained optimization that is related to the iterative \emph{triangle-fixing} algorithms that Dhillon et al.~developed for metric nearness~\cite{dhillon2004tfa}. Our approach applies broadly to any metric-constrained linear or quadratic program, and also comes with a more robust stopping criterion for the underlying iterative procedure it employs. This stopping criterion is based on a careful consideration of the dual objective function and a rounding step that is applied when our solver is close to convergence. This leads to significantly stronger constraint satisfaction and output guarantees. We additionally provide several novel results and strategies for setting a key parameter, $\gamma$, which governs the relationship between a metric-constrained linear program and the related quadratic program that can be more easily solved in practice using our iterative procedure.

We demonstrate the success of our metric-constrained optimization framework by obtaining high-quality solutions to convex relaxations of NP-hard clustering objectives on a much larger scale than has previously been accomplished in practice. In particular, we solve the Leighton-Rao relaxation of sparsest cut on graphs with thousands of nodes, as well as the LP relaxation of correlation clustering on real world networks with over 11 thousand nodes. In other words, we are able to use our techniques to solve optimization problems with $10^8$ variables and $7 \times 10^{11}$ constraints.

\section{Background and Related Work}
We briefly outline three areas in machine learning and optimization that are closely related to our work on developing solvers for metric-constrained optimization problems.

\subsection{Metric Learning and Metric Nearness} 

The problems we study are tangentially related to distance metric learning~\cite{xingdml,bellet2013surveymetric}, in which one is given a set of points in a metric space and the goal is to learn a new metric that additionally respects a set of must-link and cannot-link constraints. Specific results within the metric learning literature involve solving a metric-constrained optimization task: Batra et al. introduced a linear program with metric constraints defined on a set of \emph{code words}~\cite{batra2008semi}. Biswas and Jacobs considered a similar quadratic program which includes a full set of $O(n^3)$ triangle inequality constraints~\cite{biswasqp}. The authors note that such a QP would be very challenging to solve in practice. Our work is even more closely related to the metric nearness problem~\cite{brickell2008metricnearness}, which seeks to learn metric distance scores that are as close as possible (with respect to an $\ell_p$ norm) to a set of non-metric dissimilarity scores. Dhillon et al.\ used Dykstra's projection method to solve $\ell_p$ metric nearness problems when $p = 1,2$ or $\infty$, and applied generalized Bregman projections for all other $p$~\cite{dhillon2004tfa}. 

\subsection{Implementing Clustering Approximation Algorithms} 

A number of triangle inequality constrained LP relaxations can be rounded to produce good approximation guarantees for NP-hard clustering problems~\cite{veldt2017lamcc,leighton1999multicommodity,chawla2015near}. However, these are rarely implemented due to memory constraints.
For correlation clustering, Wirth noted that the LP can be solved more efficiently by using a multicommodity flow formulation of the problem, though this is still very expensive in practice~\cite{wirth2004approximation}. Gael and Zhu employed an \emph{LP chunking} technique which allowed them to solve the correlation clustering relaxation on graphs with up to nearly 500 nodes~\cite{van2007correlation}.
The LP rounding algorithm Charikar et al.~\cite{charikar2005clustering} inspired others to use metric-constrained LPs for modularity clustering~\cite{Agarwal2008metricmod} and joint-clustering of image segmentations~\cite{Vitaladevuni2010coclustering,glasner2011contour}. In practice these algorithms scaled to only a few hundred nodes when a full set of $O(n^3)$ constraints was included. In the case of sparsest cut, Lang and Rao developed a practical algorithm closely related to the original Leighton-Rao algorithm~\cite{lang1993finding}, which was later evaluated empirically by Lang et al.~\cite{lang2009empirical}. However, the algorithm only heuristically solves the underlying multicommodity flow problem, and therefore doesn't satisfy the same theoretical guarantees. 

\subsection{Projection Methods for Optimization}

The algorithmic framework we develop for solving metric-constrained optimization problems is based on a well-known projection method developed by Dykstra~\cite{dykstra1983algorithm}. Dykstra's method is also equivalent to Han's method~\cite{han1988successive}, and equivalent to Hildreth's method in the case of half-space constraints~\cite{hildreth1957quadratic}. Certain variants of the problem are also equivalent to performing coordinate descent on a dual objective function~\cite{tibshirani2017dykstra}. For detailed convergence results and empirical evaluations of different projection algorithms, see~\cite{bauschke2015projection,escalante2011altproj,cegielski2012iterative,censor2006computational}.

\section{Metric-Constrained LPs and Graph Clustering}
Formally, we define a metric-constrained optimization problem to be an optimization problem involving constraints of the form $x_{ij} \leq x_{ik} + x_{ik}$, where $x_{ij}$ is a non-negative variable representing the learned distance between two objects $i$ and $j$ in a given dataset. Optimization problems of this form arise very naturally in the study of graph clustering objectives, since any non-overlapping clustering $\mathcal{C}$ for a graph $G = (V,E)$ is in one-to-one correspondence with a set of binary variables $\vx = (x_{ij})$ satisfying triangle inequality constraints:
\[ \begin{cases} x_{ij} \in \{ 0, 1\}  & \text{ for all $i,j$ and } \\
x_{ij} \leq x_{ik} + x_{jk} & \text{ for all $i,j,k$} \end{cases} \iff \exists \text{ $\mathcal{C}$ s.t. }  x_{ij} = \begin{cases} 0 & \text{ if $i,j$ are together in $\mathcal{C}$} \\ 1 & \text{ otherwise}. \end{cases}   \]
In this section, we specifically consider a number of metric-constrained \emph{linear} programs, most of which arise as a relaxation of an NP-hard graph clustering task. We also prove a new equivalence between the metric nearness and the correlation clustering LPs (Theorem~\ref{thm:metric-cc}).

\subsection{Metric Nearness}
The {Metric Nearness Problem}~\cite{brickell2008metricnearness} seeks the nearest \emph{metric} matrix~$\textbf{X}^* = (x_{ij}^*)$ to a \emph{dissimilarity} matrix~$\textbf{D}=(d_{ij})$. Here, a dissimilarity matrix is a non-negative, symmetric, zero-diagonal matrix, and a \emph{metric} matrix is a dissimilarity matrix whose entries satisfy metric constraints. If~$M_n$ represents the set of metric matrices of size $n \times n$, then the problem can be formalized as follows:
\begin{equation}
\label{mnp}
\textbf{X}^* = \mbox{argmin}_{\textbf{X} \in M_n} \left(\sum_{i\neq j} w_{ij} \left |(x_{ij} - d_{ij})\right|^p   \right)^{1/p}
\end{equation}
where $w_{ij} \geq 0$ is a weight indicating how strongly we wish $\textbf{X}^*$ and $\mD$ to coincide at entry $ij$. When $p = 1$, the problem can be cast as a linear program by introducing variables $\mM = (m_{ij})$:

\begin{equation}
\label{mn1}
\begin{array}{lll} \text{minimize} & \sum_{i < j} w_{ij} m_{ij} & \\ \subjectto  &  x_{ij} \leq x_{ik} + x_{jk} & \text{ for all $i,j,k$} \\ & x_{ij} - d_{ij} \leq m_{ij} & \text{ for all $i,j$} \\ & d_{ij}-x_{ij} \leq m_{ij} &\text{ for all $i,j$} \end{array}
\end{equation}
where the last two constraints ensure that at optimality, $m_{ij} = |x_{ij} - d_{ij}|$.

\subsection{Correlation Clustering} 
{Correlation clustering} is an NP-hard problem for partitioning a signed graph $G = (V,W^+,W^-)$~\cite{Bansal2004correlation,CCEncy}. Each pair of distinct nodes $i$ and $j$ in $G$ possesses two non-negative weights, $w_{ij}^+ \in W^+$ and $w_{ij}^- \in W^-$, indicating measures of similarity and dissimilarity, respectively.
The goal is to cluster the nodes in a way that minimizes mistakes, where the mistake at pair $(i,j)$ is $w_{ij}^+$ if $i$ and $j$ are separated, and $w_{ij}^-$ if they are together. The objective can be written formally as an integer linear program:
\begin{equation}
\begin{array}{lll} \text{minimize}  & \sum_{i<j} w_{ij}^+ x_{ij} + w_{ij}^- (1-x_{ij}) &\\ \subjectto  & x_{ij} \leq x_{ik} + x_{jk} & \text{ for all $i,j,k$} \\ & x_{ij} \in \{0,1\} & \text{ for all $i,j$.} \end{array}
\label{eq:cc}
\end{equation}
An equivalent problem is to maximize the weight of agreements, which is the same at optimality but different in terms of approximation algorithms~\cite{Bansal2004correlation}. When we relax the above ILP by replacing $x_{ij} \in \{0,1\}$ with the constraint $x_{ij} \in [0,1]$, this becomes a metric-constrained linear program that has been extensively studied in the literature. Semidefinite programming relaxations for correlation clustering have also been studied~\cite{charikar2005clustering, swamy2004correlation}, and many heuristic algorithms have also been developed. However, the the best approximation results for minimizing disagreements in both general weighted graphs (an $O(\log n)$ approximation~\cite{charikar2005clustering,Demaine2003,Emanuel2003}), and complete unweighted graphs (an approximation slightly better than 2.06~\cite{chawla2015near}) depend on first solving the LP relaxation. The best results for special weighted cases and deterministic pivoting algorithms also rely on solving the LP relaxation~\cite{veldt2017lamcc, zuylen2009deterministic,puleo2016cc,puleo2015cc}.

Our first theorem shows that LP~\eqref{eq:cc} and LP~\eqref{mn1} are in fact equivalent.
\begin{theorem} 
	\label{thm:metric-cc}
	Consider an instance of correlation clustering $G = (V,W^+,W^-)$ and set $w_{ij} = |w_{ij}^+ - w_{ij}^-|$. Define an $n \times n$  matrix $\mD = (d_{ij})$ where $d_{ij} = 1$ if $w_{ij}^- > w_{ij}^+$ and $d_{ij} = 0$ otherwise. Then $\mX = (x_{ij})$ is an optimal solution to the LP relaxation of~\eqref{eq:cc} if and only if $(\mX, \mM)$ is an optimal solution for~\eqref{mn1}, where $\mM = (m_{ij}) = ( |x_{ij} - d_{ij}| )$.
\end{theorem}

\proof
	We will assume that at most one of $w_{ij}^+, w_{ij}^-$ is positive, so every pair of nodes is either labeled similar or dissimilar. If this were not the case, we could alter the edge weights so that this holds without changing the LP solution\footnote{Changing the edge weights would lead to a new instance of correlation clustering that has the same set of optimal clusterings, but may not be the same in terms of approximations. Because here we are concerned with optimally solving the LP relaxation, we can safely assume only one of $(w_{ij}^+, w_{ij}^-)$ is positive.}.
	
	We equivalently consider an unsigned graph $G' = (V,E)$ with the same node set $V$ and an adjacency matrix $\mA = (A_{ij})$ where $A_{ij} = 1$ if $w_{ij}^- = 0$ and $A_{ij}= 0$ otherwise. If $w_{ij} = \max \{ w_{ij}^+, w_{ij}^-\}$, the correlation clustering LP can then be written
	\begin{equation}
	\begin{array}{lll} \text{minimize}  &\sum_{i < j} w_{ij} \left( A_{ij}x_{ij} + (1-A_{ij}) (1-x_{ij}) \right)&\\ \subjectto  & x_{ij} \leq x_{ik} + x_{jk} & \text{ for all $i,j,k$} \\ & 0 \leq x_{ij} \leq 1 & \text{ for all $i,j$.} \end{array}
	\label{eq:cc2}
	\end{equation}
	In order to see the equivalence between the correlation clustering LP relaxation and the metric nearness problem, we
	define a dissimilarity matrix $\textbf{D} = (d_{ij})$ by setting $d_{ij} = 1- A_{ij}$. Notice that because $d_{ij} \in
	\{0,1\}$ and $x_{ij} \in [0,1]$, the key factor in the objective can be simplified thus:
	\[ (1-d_{ij}) x_{ij} + d_{ij}(1-x_{ij}) = | x_{ij} - d_{ij}|\,,\]
	and the LP relaxation of correlation clustering shown in~\eqref{eq:cc2} is equivalent to
	\begin{equation}
	\begin{array}{lll} \text{minimize}  & \sum_{i<j} w_{ij} |x_{ij} - d_{ij} | &\\ \subjectto  & x_{ij} \leq x_{ik} + x_{jk} & \text{ for all $i,j,k$} \\ & 0 \leq x_{ij} \leq 1 & \text{ for all $i,j$.} \end{array}
	\label{eq:cc3}
	\end{equation}
		The only difference between this objective and $\ell_1$ metric nearness is that we have included explicit bounds on the variables $x_{ij}$. To finish the proof we note that even without the constraint family ``$0 \leq x_{ij} \leq 1$ for all pairs $(i,j)$'',
		every optimal solution to problem~\eqref{eq:cc3} in fact satisfies those constraints. The proof of this fact is somewhat tedious; we give details in Appendix~\ref{app1}.
\endproof

\subsection{Sparsest Cut}
The \textbf{sparsest cut} score of a set $S \subset V$ in an $n$-node graph $G = (V,E)$ is defined to be
\begin{equation*}
\phi(S) = \frac{\cut(S)}{|S|} + \frac{\cut(S)}{|\bar{S}|} = \frac{n \cut(S)}{|S| |\bar{S}|},
\end{equation*}
where $\bar{S} = V\backslash S$ is the complement of $S$ and $\cut(S)$ indicates the number of edges crossing between $S$ and $\bar{S}$. Leighton and Rao developed an $O(\log n)$-approximation for finding the minimum sparsest cut set  $\phi^* = \min_{S\subset V} \phi(S)$ for any graph by solving a maximum multicommodity flow problem~\cite{leighton1999multicommodity}. This result is equivalent to solving the LP relaxation for the following metric-constrained linear program:
\begin{equation}
\label{sclp}
\begin{array}{lll} \text{minimize} & \sum_{(i,j)\in E} x_{ij}\\ \subjectto & \sum_{i<j} x_{ij} = n & \\ &x_{ij} \leq x_{ik} + x_{jk} &  \text{ for all $i,j,k$} \\ & x_{ij} \geq 0 &\text{ for all $i,j$} \end{array}
\end{equation}
and rounding the solution into a cut. The Leighton-Rao $O(\log n)$ approximation for sparsest cut was for many years the best approximation for this problem until Arora et al.\ developed an $O(\sqrt{\log(n)})$ approximation based on an SDP relaxation~\cite{arora2009expander}.


\subsection{Maximum Modularity Clustering}

Maximum modularity clustering~\cite{newman2004modularity,newman2006finding} takes a graph $G = (V,E)$ as input and seeks to optimize the following objective score over all clusterings $\mathcal{C}$:
\begin{equation}
\label{mod}
\max \,\, \frac{1}{2|E|} \sum_{i,j} \left(A_{ij} - \frac{d_id_j}{2|E|} \right)\delta^\mathcal{C}_{ij}
\end{equation}
where $d_i$ is the degree of node $i$, and $A_{ij}$ is the $\{0,1\}$ indicator for whether $i,j$ are adjacent in $G$. The $\delta^\mathcal{C}_{ij}$ variables encode the clustering:
\[ \delta^\mathcal{C}_{ij} = \begin{cases} 1 & \text{ if $i,j$ are together in $\mathcal{C}$} \\ 0 & \text{ otherwise}. \end{cases} \]
Although modularity has been widely used in clustering applications, Dinh showed that it is NP-hard to obtain a constant-factor approximation algorithm for the objective~\cite{dinh2016network}. However, inspired by the LP rounding algorithm of Charikar et al.\ for correlation clustering~\cite{charikar2005clustering}, Agarwal and Kempe noted that by replacing $\delta^\mathcal{C}_{ij} = 1 - x_{ij}$ in~\eqref{mod} and introducing metric constraints, one obtains a the following LP relaxation~\cite{Agarwal2008metricmod}:
\begin{equation}
\label{modlp}
\begin{array}{lll} \text{maximize} & \frac{1}{2|E|} \sum_{i,j} \left(A_{ij} - \frac{d_id_j}{2|E|} \right) (1-x_{ij})\\ \subjectto &x_{ij} \leq x_{ik} + x_{jk} & \text{ for all $i,j,k$} \\& 0 \leq x_{ij} \leq 1 &\text{ for all $i,j$.} \end{array}
\end{equation}
Solving this LP relaxation provides a useful upper bound on the maximum modularity. In practice this opens up the possibility of obtaining a posteriori approximation guarantees for using heuristic methods for modularity clustering.

\subsection{Cluster Deletion}
Cluster deletion is the problem of deleting a minimum number of edges in an unweighted, undirected graph $G = (V,E)$ so that the remaining graph is a disjoint set of cliques. This problem can be viewed as a variant of correlation clustering in which each pair of nodes has weights $(w_{ij}^+ = 1, w_{ij}^- = 0)$ or $(w_{ij}^+ = 0, w_{ij}^- = \infty)$. The LP relaxation can be obtained by starting with the relaxation for correlation clustering and fixing $x_{ij} =1$ for $(i,j) \notin E$. This eliminates the need to work explicitly with variables for non-edges, and therefore simplifies into the following relaxation:
\begin{equation}
\label{cdlp}
\begin{array}{lll} \text{minimize} & \sum_{(i,j)\in E} x_{ij}\\ \subjectto &x_{ij} \leq x_{ik} + x_{jk} & \text{ if $(i,j,k) \in T$} \\ &1 \leq x_{ik} + x_{jk} & \text{ if $(i,j,k) \in \tilde{T}_k$} \\& 0 \leq x_{ij} \leq 1 &\text{ for all $(i,j) \in E$}. \end{array}
\end{equation}
In the above, $T$ represents the set of triangles, i.e. triplets of nodes $i,j,k$ that form a clique in $G$. The set $\tilde{T}_k$ represents \emph{bad} triangles ``centered" at $k$, i.e. $k$ shares an edge with $i$ and $j$, but $(i,j) \notin E$. In recent work we showed that the solution to this LP can be rounded into a clustering that is within a factor two of the optimal cluster deletion solution~\cite{veldt2017lamcc}.


\subsection{Maximum Cut}
Given a graph $G = (V,E)$, the maximum cut problem seeks to partition $G$ into two clusters in a way that maximizes the number of edges crossing between the clusters. The linear programming relaxation for Max Cut is
\begin{equation}
\label{maxcut}
\begin{array}{lll} \text{maximize} & \sum_{i,j} A_{ij} x_{ij}\\ \subjectto &x_{ij} \leq x_{ik} + x_{jk} & \text{ for all $i,j,k$} \\&x_{ij} + x_{ik} +x_{jk} \leq 2  & \text{ for all $i,j,k$} \\ &0 \leq x_{ij} \leq 1 &\text{ for all $i,j$.} \end{array}
\end{equation}
Note that if the linear constraints $x_{ij} \in [0,1]$ were replaced with binary constraints $x_{ij} \in \{0,1\}$, then this would correspond to an integer linear program for the exact Max Cut objective. The constraint $x_{ij} + x_{ik} +x_{jk} \leq 2$ are included to ensure in the binary case that nodes are assigned to at most two different clusters. It is well known that the integrality gap of this linear program is $2-\epsilon$~\cite{poljak1995maximum}. Fernandez de la Vega and Mathieu noted that this can be improved to a $1+\epsilon$ integrality gap in the case of dense graphs when additional constraints are added~\cite{delaVega2007maxcut}.

\section{Projection Methods for Quadratic Programming}
\label{algorithm}
The graph clustering problems we have considered above can be relaxed to obtain an LP of the form
\begin{equation}
\min_{\vx} \,\, \vc^T \vx \hspace{.2cm} \text{ s.t. }  \mA \vx \leq \vb,
\label{lp}
\end{equation}
where $\mA$ is a large and very sparse matrix with up to $O(n^3)$ rows and $O(n^2)$ columns. Standard optimization software will be unable to solve these LPs for large values of $n$, due to memory constraints, so we instead turn our attention to applying a simple projection method for solving a closely related quadratic program:
\begin{equation}
\min_{\vx} \,\, Q(\vx) = \vc^T \vx + \frac{1}{2\gamma} \vx^T \mW \vx \hspace{.2cm} \text{ s.t. }  \mA \vx \leq \vb,
\label{qp}
\end{equation}
where $\mW$ is a diagonal matrix with positive diagonal entries and $\gamma > 0$.  When $\mW$ is the identity matrix, it is well known that there exists some $\gamma_0 > 0$ such that for all $\gamma > \gamma_0$, the optimal solution to the quadratic program corresponds to the minimum 2-norm solution of the original LP~\cite{mangasarian1984normal}. 

\subsection{Applying Dykstra's Projection Method}
The dual of~\eqref{qp} is another quadratic program:
\begin{align}
\max_{\vy} \,\, D(\vy) = -\vb^T \vy - &\frac{1}{2\gamma}({\textbf{A}^T \vy + \vc })^T\mW^{-1}({\textbf{A}^T \vy + \vc }) \hspace{.2cm} \text{ s.t. }  \vy \geq 0. \label{dual1}
\end{align}
The core of our algorithm for solving~\eqref{qp} is Dykstra's projection method~\cite{dykstra1983algorithm}, which iteratively updates a set of primal and dual variables $\vx$ and $\vy$ that are guaranteed to converge to the optimal solution of~\eqref{qp} and~\eqref{dual1} respectively. 
The method cyclically visits constraints one at a time, first performing a \emph{correction} step to the vector $\vx$ based on the dual variable associated with the constraint, and then performing a \emph{projection} step so that $\vx$ satisfies the linear constraint in question. Pseudocode for the method, specifically when applied to our quadratic program, is given in Algorithm~\ref{dykstra}. For quadratic programming, Dykstra's method is also equivalent to Hildreth's projection method~\cite{hildreth1957quadratic}, and is guaranteed to have a linear convergence rate~\cite{escalante2011altproj}. In Appendix~\ref{app2}, we provide more extensive background information regarding Dykstra's method. In particular we prove the equivalence relationship between Dykstra's method and Hildreth's method in the case of quadratic programming. We also provide a slight generalization of the results of Dax~\cite{simplealg} which show that the dual variables produced by Dykstra's method allow us to obtain a strictly increasing lower bound on the quadratic objective~\eqref{qp} we are trying to solve.

\begin{algorithm}[tb]
	\caption{Dykstra's Method for Quadratic Programming}
	\begin{algorithmic}[5]
		\State \textbf{Input:} $\mA  \in \mathbb{R}^{N\times M}, \vb \in \mathbb{R}^M, \vc \in \mathbb{R}^N, \gamma > 0, \mW \in \mathbb{R}^{N\times N} (\text{diagonal, positive definite})$ 
		\State \textbf{Output:} $\hat{\vx} = \argmin_{\vx \in \mathcal{A}} Q(\vx)$ where $\mathcal{A} = \{ \vx \in \mathbb{R}^N: \mA \vx \leq \vb \}$ 
		\State $\vy := \textbf{0} \in \mathbb{R}^M$ 
		\State $\vx := -\gamma\mW^{-1}\vc$, $k := 0$
		\While{\emph{not converged}}
		\State $k := k+1$
		\State (Visit constraints cyclically): $i := (k-1) \bmod M + 1$ 
		\State (Perform correction step):  $\vx := \vx + y_i (\gamma\mW^{-1} \va_i)$ 
		where $\va_i$ is the $i$th row of $\mA$
		\State (Perform projection step): $\vx := \vx - \theta_i^+ (\gamma\mW^{-1} \va_i)$
		where $\theta_i^+ = \frac{ \max \{\va_i^T \vx - b_i, 0 \}}{\gamma \va_i^T \mW^{-1} \va_i }$
		\State (Update dual variables): $y_i := \theta_i^+ \geq 0$
		\EndWhile
	\end{algorithmic}
	\label{dykstra}
\end{algorithm}

\section{Implementing Dykstra's Method for Metric-Constrained Optimization}
\label{details}
Our full algorithmic approach takes Dykstra's method (Algorithm~\ref{dykstra}) and includes a number of key features that allow us to efficiently obtain high-quality solutions to metric-constrained problems in practice. The first feature, a procedure for locally performing projections at metric constraints, is the key insight which led Dhillon et al.\ to develop efficient algorithms for metric nearness~\cite{dhillon2004tfa}. In addition, we detail a sparse storage scheme for dual vectors, and include a more robust convergence check that leads to better constraint satisfaction and stronger optimality guarantees for a variety of metric-constrained problems.

In order to demonstrate our application of Dykstra's method to metric-constrained linear programs, we will specifically consider the quadratic program related to the Leighton-Rao sparsest cut relaxation:
\begin{equation}
\label{sclp2}
\begin{array}{lll} \text{minimize} & \sum_{(i,j)\in E} x_{ij} + \frac{1}{2\gamma}\sum_{i<j} w_{ij} x_{ij}^2 \\ \subjectto & \sum_{i<j} x_{ij} = n & \\ &x_{ij} \leq x_{ik} + x_{jk} &  \text{ for all $i,j,k$} \\ & x_{ij} \geq 0 &\text{ for all $i,j$} \end{array}
\end{equation}
where $w_{ij} = 1$ if $(i,j) \in E$ and $w_{ij} = \lambda$ for some $\lambda \in (0,1)$ otherwise. We give justification for this choice of weights matrix in Section~\ref{sec:guarantees}. To slightly simplify expressions later in this section we will the parameter $\gamma$ directly into a new weight matrix $\mW_\gamma = \mW/\gamma$. Initially the vector of dual variables $\vy$ is set to zero, and $\vx = -\mW_\gamma^{-1} \vc$. For the sparsest cut relaxation in particular this means we set $\vx = (x_{ij})$ as follows:
\[ x_{ij}=  \begin{cases} -\gamma & \text{if $(i,j) \in E$ }\\
0 & \text{ otherwise}.
\end{cases}\]

\subsection{Efficient local updates}
Projections of the form $\vx := \vx + \alpha \mW^{-1} \va_i$
for $\mW$ diagonal and a constant~$\alpha$ will change~$\vx$ by at most the number of nonzero entries of $\va_i$, the $i$th row of constraint matrix $\mA$. In the case of triangle inequality constraints, which dominate our constraint set, there are exactly three non-zero entries per constraint, so we perform each projection in a constant number of operations. 

Consider the triangle inequality constraint $x_{ij} - x_{ik} - x_{jk} \leq 0$ associated with an ordered triplet $(i,j,k)$. Let $t = t_{ijk}$ represent a unique ID corresponding to this constraint. For now we ignore the correction step of Dykstra's method, which is skipped over in the first round since the vector of dual variables is initialized to zero (i.e. $y_t = 0$). The projection step we must perform is
\[ \vx \leftarrow \vx -  \frac{ [\va_t^T \vx - b_t]^+}{\va_t^T \mW_\gamma^{-1} \va_t } \mW_\gamma^{-1} \va_t \]
where $\va_t$ contains exactly three entries: $1$, $-1$, and $-1$, at the locations corresponding to variables $x_{ij}$, $x_{ik}$, and $x_{jk}$. This projection will only change $\vx$ if constraint $t$ is violated, so we first check if
\[ \delta = \va_t^T \vx - b_t = x_{ij} - x_{ik} - x_{jk} > 0. \]
If so, we compute
\[ \theta_t^+ = \frac{ [\va_t^T \vx - b_t]^+}{\va_t^T \mW_\gamma^{-1} \va_t } = \frac{\delta}{1/w_{ij} + 1/w_{ik} + 1/w_{jk}} = \frac{\delta w_{ij}w_{ik}w_{jk}}{w_{ij}w_{ik}+w_{ij}w_{jk} + w_{ik}w_{ij}}. \]
The projection step then updates exactly three entries of $\vx$:
\[ x_{ij} \leftarrow x_{ij} - \theta_t^+\frac{ x_{ij}}{w_{ij}}, \hspace{1cm} x_{ik} \leftarrow x_{ik} - \theta_t^+\frac{x_{ik}}{w_{ik}}, \hspace{1cm} x_{jk} \leftarrow x_{jk} - \theta_t^+  \frac{x_{jk}}{w_{jk} }. \]
Note that all of this can be done in a constant number of operations for each triangle inequality constraint.

\subsection{Sparse storage of $\vy$}
For constraint sets that include triangle inequalities for every triplet of nodes $(i,j,k)$, the dual vector $\vy$ will be of length $O(n^3)$. Observe that the correction step in Algorithm~\ref{dykstra} at constraint $t$ will be nontrivial if and only if there was a nontrivial projection last time the constraint was visited. In other words, $\theta_t^+$ was nonzero in the previous round and therefore $y_t > 0$. 

\subsubsection{Sparsity in the Triangle Constraint Variables.} Note that each triplet $(i,j,k)$ corresponds to three different metric constraints: $x_{ij} - x_{ik} - x_{jk} \leq 0$, $x_{jk} - x_{ik} - x_{ij} \leq 0$, and $x_{ik} - x_{ij} - x_{jk} \leq 0$, and in each round at most one of these constraints will be violated, indicating that at least two dual variables will be zero. Dhillon et al.\ concluded that ${n \choose 3}$ floating point numbers must be stored in implementing Dykstra's algorithm for the metric nearness problem~\cite{dhillon2003MNreport}. We further observe, especially for the correlation clustering LP, that often in practice for a large percentage of triplets $(i,j,k)$, none of the three metric constraints is violated. Thus we can typically avoid the worst case $O(n^3)$ memory requirement by storing $\vy$ sparsely. 

\subsubsection{Storing $\vy$ in dictionaries or arrays.}
Conceptually the easiest approach to storing nonzero entries in $\vy$ is to maintain a dictionary of key-value pairs $(t,y_t)$. In this case, when visiting constraint $t$, we check if the dictionary contains a nonzero dual variable $y_t > 0$ for this constraint, and if so we perform the corresponding non-trivial correction step. However, because we always visit constraints in the same order, we find it faster in practice to store two arrays with pairs $(t,y_t)$ rather than a dictionary. The first array stores entries $y_t$ that were set to a nonzero value in the previous pass through the constraints. We maintain a pointer to the entry in the array which gives us the next such constraint $t$ that will require a nonzero correction in the current pass through the constraint set. The second array allocates space for the new dual variables that become nonzero after a projection step in the {current} pass through the constraints.  These will be needed for corrections in the next round. Dykstra's method does not require we remember history beyond the last pass through constraints, so we never need more than two arrays storing pairs $(t,y_t)$.

\subsubsection{Pseudocode}
Algorithm~\ref{metproj} displays pseudocode for one step of our implementation of Dykstra's method when visiting a metric constraint. We assume the nonzero dual variables $y_t$ are stored sparsely and can be efficiently queried and updated. The same basic outline applies also to non-metric constraints.

\begin{algorithm}[tb]
	\caption{MetricProjection$(i,j,k)$}
	\begin{algorithmic}[5]
		\State $t := \text{unique ID for $(i,j,k)$}$
		\State Obtain $(x_{ij},x_{ik}, x_{jk})$ and weights $(w_{ij},w_{ik}, w_{jk})$ from $\mX$ and $\mW_\gamma$
		\If{$y_t > 0$}
		\State $x_{ij} \leftarrow x_{ij} + y_t\frac{ x_{ij}}{w_{ij}}, \hspace{.1cm} x_{ik} \leftarrow x_{ik} + y_t\frac{x_{ik}}{w_{ik}}, \hspace{.1cm} x_{jk} \leftarrow x_{jk} + y_t  \frac{x_{jk}}{w_{jk} }$.
		\EndIf
		\State $\delta := x_{ij} - x_{ik} - x_{jk}$
		\If{$\delta > 0$}
		\State $\theta_t = \frac{\delta w_{ij}w_{ik}w_{jk}}{w_{ij}w_{ik}+w_{ij}w_{jk} + w_{ik}w_{ij}}$
		\State $x_{ij} \leftarrow x_{ij} - \theta_t\frac{ x_{ij}}{w_{ij}}, \hspace{.1cm} x_{ik} \leftarrow x_{ik} - \theta_t\frac{x_{ik}}{w_{ik}}, \hspace{.1cm} x_{jk} \leftarrow x_{jk} - \theta_t  \frac{x_{jk}}{w_{jk} }$.
		\State Store $y_t = \theta_t$
		\EndIf
	\end{algorithmic}
	\label{metproj}
\end{algorithm}

\subsection{Robust Stopping Criteria}
Many implementations of Dykstra's method stop when the change in vector $\vx$ drops below a certain tolerance after one or more passes through the entire constraint set. However, Birgin et al. noted that in some cases this may occur even when the iterates are far from convergence~\cite{birgin2005robust}. Because we are applying Dykstra's method specifically to quadratic programming, we can obtain a much more robust stopping criterion by carefully monitoring the dual objective function and dual variables, in a manner similar to the approach of Dax~\cite{simplealg}. 

\subsubsection{Optimality Conditions} Let $(\vy_k, \vx_k)$ denote the pair of primal and dual vectors computed by Dykstra's method after $k$ projections. We know that these vectors will converge to an optimal pair $(\hat{\vx},\hat{\vy})$ such that $D(\hat{\vy}) = \hat{Q} = Q(\hat{\vx})$ where $\hat{Q}$ is the optimal objective for both the primal~\eqref{qp} and dual~\eqref{dual1} quadratic programs. The KKT optimality conditions for quadratic programming state that the pair $(\hat{\vx}, \hat{\vy})$ is optimal for the primal~\eqref{qp} and dual~\eqref{dual1} quadratic programs if and only if the following conditions hold:
\[ \textbf{1. $\mA \hat{\vx} \leq \vb$ } \hspace{1cm} \textbf{2. $\hat{\vy}^T (\mA \hat{\vx}-\vb) = 0$ } \hspace{1cm} \textbf{3. ${\mW_\gamma}\hat{ \vx} = -\textbf{A}^T \hat{\vy} -\vc$. } \hspace{1cm} \textbf{4. $\hat{\vy} \geq 0$}  \]
In this case we know that $D(\hat{\vy}) = \hat{Q} = Q(\hat{\vx})$ where $\hat{Q}$ is the optimal objective for both the primal and dual quadratic programs. We show iin Appendix~\ref{app2} that the dual update step $y_i := \theta_i^+$ in Algorithm~\ref{dykstra} will guarantee that the last two KKT conditions are always satisfied. In other words, the primal and dual variables at iteration $k$, $(\vx_k ,\vy_k)$, satisfy $\vy_k \geq 0$ and ${\mW_\gamma} \vx_k = -\textbf{A}^T \vy_k -\vc$. This means that $\vy_k$ is always feasible for the dual objective~\eqref{dual1}, and by weak duality we have the following lower bound on the optimal solution to objective~\eqref{qp}
\begin{equation}
\label{eq:dbound}
D(\vy_k)= -\vb^T \vy_k- \frac{1}{2\gamma}(\textbf{A}^T \vy_k + \vc )^T\mW^{-1}_\gamma(\textbf{A}^T \vy_k + \vc ) = -\vb^T\vy_k - \frac{1}{2} \vx_k^T \mW_\gamma \vx_k.
\end{equation}
We also prove in the Appendix that performing Dykstra's method is equivalent to applying a coordinate ascent procedure on the dual quadratic program~\eqref{dual1}. This means that $D(\vy_k)$ is a strictly increasing lower bound that converges to $\hat{Q}$. Meanwhile, $Q(\vx_k)$ does not necessarily upper bound $\hat{Q}$ since $\vx_k$ is not necessarily feasible. However, $\vx_k$ converges to the optimal primal solution, so as the algorithm progresses, the maximum constraint violation of~$\vx_k$ decreases to zero. In practice, once $\vx_k$ has satisfied constraints to within a small enough tolerance we treat $Q(\vx_k)$ as an upper bound. 

After each pass through the constraints we check the primal-dual gap~$\omega_k$ and maximum constraint violation~$\rho_k$, given by
\begin{align*}
\omega_k &= \frac{D(\vy_k) - Q(\vx_k)}{D(\vy_k)} \\
\rho_k &= \max_t (b_t - \va_t^T \vx_k).
\end{align*}
Together these two scores provide an indication for how close $(\vx_k, \vy_k)$ are to convergence.

\subsubsection{Computing $\omega_k$ and $\rho_k$}
To compute $\omega_k$ in practice, we note that $\frac{1}{2} \vx_k^T \mW_\gamma \vx_k$ appears in both $Q(\vx_k)$ and $D(\vy_k)$ (see~\eqref{eq:dbound}). This term, as well as the term $\vc^T \vx$ can be easily computed by iterating through the $O(n^2)$ entries of $\vx$. Finding $\vb^T \vy_k$ could theoretically involve $O(n^3)$ computations, but this can be done during the main loop of Dykstra's algorithm by continually updating a variable that adds up terms of the form $y_tb_t$ whenever $y_t$ and $b_t$ are both nonzero for a particular constraint $t$. Note that in most of the problems we have considered here, $b_t = 0$ for the majority of the constraints. For example, in the sparsest cut relaxation~\eqref{sclp2}, $b_t$ is only nonzero for the constraint $\sum_{i<j} x_{ij} = n$.

Computing $\rho_k$ requires we iterate though the entire constraint set and simply record the worst constraint violation. Since this requires visiting $O(n^3)$ constraints, it may take nearly as long as a full pass through constraints using Dykstra's method. In practice we therefore just check each constraint until we come across one that violates the desired constraint tolerance, if such a constraint exists. At this point we know the algorithm did not converge, and there is no need to continue checking constraint violations. Every 10-20 passes through the algorithm we perform a full constraint check to report on the progress of the algorithm.

\subsection{Entrywise Rounding Procedure} In practice we could simply run Dykstra's iteration until both~$\omega_k$ and~$\rho_k$ fall below user-defined tolerances. We additionally incorporate another step in out convergence check that significantly improves the algorithm's performance in practice. Because~$\vx_k \rightarrow \hat{\vx}$, we know that after a certain point, the maximum entrywise difference between~$\hat{\vx}$ and~$\vx_k$ will be arbitrarily small. Therefore, once both $\rho_k$ and $|\omega_k|$ have dropped below a given tolerance, we will test for convergence by rounding every entry of $\vx_k$ to $r$ significant figures for a range of values of $r$: $\vx_r = round(\vx_k,r)$. As long as $\vx_k$ is close enough to optimality and we have chosen the proper $r$, $\vx_r$ will satisfy constraints to within the desired tolerance and will have an objective exactly or nearly equal to the best lower bound we have for $\hat{Q}$: $[D(\vy_k) - Q(\vx_r)]/D(\vy_k) \leq \epsilon$. If $\vx_r$ does not satisfy constraints or has a poor objective score, we simply discard $\vx_r$ and continue with $\vx_k$ and the original Dykstra iteration. Even if this rounding procedure is always unsuccessful, we simply fall back on the iterates $(\vx_k,\vy_k)$ until $~\omega_k$ and $\rho_k$ eventually fall below the defined tolerance. In practice however, we do find that the rounding procedure dramatically improves both the runtime of our method as well as constraint satisfaction.

We highlight the fact that when checking whether we are close enough to convergence to apply the entrywise rounding step, we consider the absolute value of $\omega_k$ and not $\omega_k$ itself. Recall that this value may be negative if $Q(\vx_k)$ is not an upper bound on the optimal objective. Often in practice we find that by the time we are close to convergence, $Q(\vx_k)$ is indeed an upper bound and $\omega_k$ is a small positive number. However, we also observe cases where $\vx_k$ is infeasible and $\omega_k$ is negative, but $|\omega_k|$ is small and our entrywise rounding procedure succeeds in producing a feasible point $\vx_r$. When this happens, the duality gap between $D(\vy_k)$ and $Q(\vx_r)$ is guaranteed to be non-negative, and tells us how close the feasible vector $\vx_r$ is to the optimal solution. 
\section{Approximation Guarantees for Clustering Objectives}
\label{sec:guarantees}
\label{approx}

The results of Mangasarian confirm that for all~$\gamma$ greater than some $\gamma_0 > 0$, the original linear
program~\eqref{lp} and the quadratic regularization~\eqref{qp} will have the same optimal
solution~\cite{mangasarian1984normal}. However, it is challenging to compute~$\gamma_0$ in practice, and if we
set~$\gamma$ to be too high then this may lead to very slow convergence for solving QP~\eqref{qp} using projection
methods. Dhillon et al.~suggest ways to set~$\gamma$ for variants of the metric nearness problem based on empirical
observations, but no approximation guarantees are provided~\cite{dhillon2004tfa}. A key contribution in our work is a
set of results, outlined in this section, which show how to set~$\mW$ and~$\gamma$ in order to obtain specific
guarantees for approximating the correlation clustering and sparsest cut objectives. These results
hold for all $\gamma > 0$, whether larger or smaller than the unknown value~$\gamma_0$. We begin with a general theorem that provides a useful strategy for obtaining approximation guarantees for a large class of linear programs.
\begin{theorem}
	\label{genthm}
	Let $\mA \in \mathbb{R}^{M\times N}$, $\vb \in \mathbb{R}^N$, $\vc \in \mathbb{R}^N_{> 0}$, and $\mathcal{A} = \{\vx \in \mathbb{R}^N : \mA \vx \leq \vb \}$. Denote $\vx^* = \argmin_{\vx \in \mathcal{A}} \vc^T \vx$ and assume that all entries of $\vx^*$ are between $0$ and $B > 0$. Let $\mW$ be a diagonal matrix with entries $\mW_{ii} = c_i > 0$ and let $\hat{\vx} =  \argmin_{\vx \in \mathcal{A}} [\vc^T \vx + 1/(2\gamma)\vx^T \mW \vx]$. Then
	\[ \vc^T \xlp \leq \vc^T \xqp \leq \vc^T \xlp ( 1+ B/(2\gamma) ). \]
\end{theorem}
\begin{proof}
	Vectors $\hat{\vx}$ and $\vx^*$ are optimal for their respective problems, meaning that
	\[ \vc^T\vx^* \leq \vc^T\hat{\vx} \leq \vc^T\hat{\vx} + \frac{1}{2\gamma} \hat{\vx}^T \mW \hat{\vx} \leq \vc^T\vx^* + \frac{1}{2\gamma} ( \vx^*)^T \mW \vx^*. \]
	The proof follows from combining these inequalities with a bound on the second term on the far right. By our construction of $\mW$ and the bounds we assume hold for $\vx^*$, we have:
	\[ (\vx^*)^T \mW \vx^* = \sum_{i=1}^n c_i (\vx_i^*)^2 = B^2 \sum_{i=1}^n c_i (\vx_i^*/B)^2 \leq B \sum_{i=1}^n c_i \vx_i^* = B \vc^T \vx^*\]
	where the second to last step holds because $0 \leq x_i^* \leq B \implies (x_i^*/B)^2 < (x_i^*/B)$.
\end{proof}

\subsection{Cluster Deletion Approximation} We observe that Theorem~\ref{genthm} directly implies a result for the cluster deletion LP relaxation~\eqref{cdlp}. Cluster deletion has a variable $x_{ij}$ for each edge $(i,j) \in E$. The objective can be written $\ve^T \vx = \sum_{(i,j) \in E} x_{ij}$, where $\ve$ is the all ones vector. Since the LP also includes constraints $x_{ij} \in [0,1]$, we see that the assumptions of Theorem~\ref{genthm} hold with $\mW$ equal to the identity matrix and $B =1$. This means that we can use our Dykstra-based projection method to optimize a quadratic objective that produces a solution within a factor $(1+ 1/(2\gamma))$ of the optimal cluster deletion LP relaxation. Coupling this result with the LP rounding procedure we developed in previous work, we can obtain a $(2 + 1/ \gamma)$ approximation for cluster deletion in practice~\cite{veldt2017lamcc}.
 
\subsection{Correlation Clustering}

Consider a correlation clustering problem on $n$ nodes where each pair of nodes $(i,j)$ is either strictly similar or
strictly dissimilar, with a nonzero weight $w_{ij} > 0$. That is, exactly one of the weights $(w_{ij}^-, w_{ij}^+)$ is positive and the other is zero.
We focus on the LP relaxation for this problem given in the form of the $\ell_1$ metric nearness LP~\eqref{mn1}. We slightly alter this formulation by performing a change of variables $y_{ij} = x_{ij} - d_{ij}$. The LP can then be written equivalently as:
\begin{equation}
\label{mncc}
\begin{array}{lll} \text{minimize} & \sum_{i < j} w_{ij} m_{ij} & \\ \subjectto  &  y_{ij} - y_{ik} - y_{jk} \leq b_{ijk}  & \text{ for all $i,j,k$} \\ & y_{ij} \leq m_{ij} & \text{ for all $i,j$} \\ & -y_{ij} \leq m_{ij} &\text{ for all $i,j$} \end{array}
\end{equation}
where $b_{ijk} = -d_{ij} + d_{ik} + d_{jk}$ is defined so that the implicit variables $x_{ij} = y_{ij} + d_{ij}$ satisfy triangle inequalities. To write this LP in the format of~\eqref{lp}, we as usual use $\vx$ to represent the set of variables of the linear program. However, for this problem we must take care to note that $\vx$ does not represent a linearization of the $x_{ij}$ distance variables, but instead stores both $y_{ij}$ and $m_{ij}$ variables. More precisely, to relate~\eqref{mncc} to the format of LP~\eqref{lp}, we set $\vx = \begin{bmatrix} \vy &  \vm \end{bmatrix}^T$ and $\vc = \begin{bmatrix} \textbf{0} & \vw \end{bmatrix}^T$, where $\vy, \vm$ represent linearizations of the doubly-indexed $(y_{ij})$ and $(m_{ij})$ variables, and $\vw = (w_{ij})$ is the vector of positive weights for the node pairs. 
Rather than minimizing $\vc^T \vx = \sum_{i<j} w_{ij} m_{ij}$ we have a method that can minimize the quadratic objective $\vc^T \vx + \frac{1}{2\gamma} \vx^T \mW \vx$ over the same constraint set. We construct a weight matrix that contains two copies of the weight vector $\vw$, one to match up with the $\vy$ vector and one corresponding to the $\vm$ vector:
\begin{equation}
\label{ccW}
\mW = \begin{bmatrix} diag(\vw) & \textbf{0} \\ \textbf{0} & diag(\vw) \end{bmatrix}.
\end{equation}
The quadratic regularization of the original LP objective is then
\begin{equation}
\label{ccqp}
\min_{\vx} \vc^T \vx + \frac{1}{2\gamma} \vx^T \mW \vx= \min_{(m_{ij}), (y_{ij})} \,\, \sum_{i<j} w_{ij} m_{ij} +\frac{1}{2\gamma} \sum_{i<j} w_{ij} m_{ij}^2 + \frac{1}{2\gamma} \sum_{i<j} w_{ij} y_{ij}^2.
\end{equation}
Observe that for the original LP~\eqref{mncc} as well as the regularized objective~\eqref{ccqp}, the variables satisfy $m_{ij} = |y_{ij}|$ at optimality. This implies that $m_{ij}^2 = y_{ij}^2$, which is the reason we choose to introduce variables $y_{ij} = x_{ij} - d_{ij}$ rather than working directly with $x_{ij}$. Introducing $y_{ij}$ variables allows us to replace $y_{ij}^2$ with $m_{ij}^2$ in~\eqref{ccqp}, and re-write the objective using terms only involving $m_{ij}$ variables:
\begin{equation}
\label{ccqp2}
\min_{(m_{ij}), (y_{ij})}    \,\, \sum_{i<j} w_{ij} m_{ij} +\frac{1}{\gamma} \sum_{i<j} w_{ij} m_{ij}^2. 
\end{equation}
Let $(m^*_{ij})$ and $(y^*_{ij})$ be optimal for~\eqref{mncc} and $(\hat{m}_{ij}), (\hat{y}_{ij})$ be optimal for~\eqref{ccqp2}. Then
\begin{equation}
\sum_{i<j} w_{ij} \hat{m}_{ij} +\frac{1}{\gamma} \sum_{i<j} w_{ij} \hat{m}_{ij}^2 \leq \sum_{i<j} w_{ij} m^*_{ij} +\frac{1}{\gamma} \sum_{i<j} w_{ij} (m^*_{ij})^2 \leq \left(1 + \frac{1}{\gamma}  \right) \sum_{i<j} w_{ij} {m}^*_{ij}.
\end{equation}
In the last step above we have used the fact that $m_{ij}^* = | y_{ij}^*| \leq 1 \implies m_{ij}^* \leq (m_{ij}^*)^2$ (see the proof of Theorem~\ref{thm:metric-cc} and Lemma~\ref{lemma1} in the Appendix for why $|y_{ij}^*| \leq 1$). This proves an approximation result for correlation clustering:
\begin{theorem}
	Let $(m^*_{ij})$ and $(y^*_{ij})$ be the optimal solution vectors for the correlation clustering LP relaxation given in~\eqref{mncc} and $(\hat{m}_{ij}), (\hat{y}_{ij})$ be the optimal solution to the related QP~\eqref{ccqp}. Then
	\[\sum_{i<j} w_{ij} m^*_{ij}  \leq \sum_{i<j} w_{ij} \hat{m}_{ij} \leq \left(1 + \frac{1}{\gamma}  \right) \sum_{i<j} w_{ij} {m}^*_{ij}. \]
\end{theorem}
Therefore, given any rounding procedure for the original LP that gives a factor $p$ approximation for a correlation clustering problem, we can instead solve the related QP using projection methods to obtain a factor $p(1+1/\gamma)$ approximation. For weighted correlation clustering, the best rounding procedures guarantee an $O(\log n)$ approximation~\cite{Demaine2003,Emanuel2003,charikar2005clustering}, so this can still be achieved even if we use a small value for $\gamma$.

\subsection{Sparsest Cut}
\label{scbound}
The Leighton-Rao linear programming relaxation for sparsest cut is presented in~\eqref{sclp}. This LP has a variable $x_{ij}$ for every pair of distinct nodes $i <j$ in some unweighted graph $G = (V,E)$. Let $\vx = (x_{ij})$ be a linearization of these distance variables, and define $\vc = (c_{ij})$ to be the adjacency indicators, i.e.
\[ c_{ij} = \begin{cases} 1 & \text{ if $(i,j) \in E$} \\
0 & \text{ otherwise}.
\end{cases}
\]
Then the objective can be written in the familiar format $\min_{\vx} \,\, \vc^T \vx$. If we assume we have chosen a weight matrix $\mW$ and a parameter $\gamma > 0$, the regularized version of~\eqref{sclp} has objective
\[ \min_{\vx} \,\, \sum_{i < j} c_{ij} x_{ij} + \frac{1}{2\gamma} \sum_{i<j} w_{ij} x_{ij}^2.  \]
As we have seen in Theorem~\ref{genthm}, it seems fitting for $c_{ij}$, the coefficients of $x_{ij}$, to match up with $w_{ij}$, the coefficients of $x_{ij}^2$. However, many of the $c_{ij}$ variables are zero, so it will not work to choose $w_{ij} = c_{ij}$, since $\mW$ needs to be positive definite in order for us to apply our projection method. Instead we introduce another parameter $\lambda \in (0,1)$ and define a set of weights $\vw = (w_{ij})$ by
\[ w_{ij} = \begin{cases} 1 & \text{ if $(i,j) \in E$} \\
\lambda & \text{ otherwise}.
\end{cases}
\]
In this way, the weight $w_{ij}$ is still positive but can be \emph{near} zero (i.e. near $c_{ij}$) when $(i,j) \notin E$. We can then prove the following approximation result:
\begin{theorem}
	\label{scthm}
	Let $G = (V,E)$ be a connected graph with $n = |V| > 4$. Let~$\phi^*$ be the minimum sparsest cut score for $G$ and assume that each side of the optimal sparsest cut partition has at least 2 nodes. Let $\gamma > 0$, $\mW = diag(\vw)$ be defined as above for a given $\lambda \in (0,1)$, and let $\mathcal{A}$ denote the set of constraints from the Leighton-Rao LP relaxation for sparsest cut. Then
	\[ \min_{\vx \in \mathcal{A}} \,\, \vc^T {\vx} + \frac{1}{2\gamma} {\vx}^T \mW{\vx} \leq \left(1 + \frac{1+ \lambda n}{2\gamma } \right) \phi^*. \]
\end{theorem}

\begin{proof}
	The quadratic regularization of the sparsest cut LP relaxation is
	\begin{equation}
	\label{scqp}
	\begin{array}{lll} \text{minimize} & \sum_{i<j} c_{ij} x_{ij} + (1/2\gamma) \sum_{i<j} w_{ij} x_{ij}^2\\ \subjectto & \sum_{i<j} x_{ij} = n & \\ &x_{ij} \leq x_{ik} + x_{jk} &  \text{ for all $i,j,k$} \\ & x_{ij} \geq 0 &\text{ for all $i,j$}. \end{array}
	\end{equation}
	The result we prove here relates the optimal solution of~\eqref{scqp} directly back to the minimum sparsest cut score $\phi^*$, rather than back to the LP relaxation of sparsest cut~\eqref{sclp}. This makes sense given that our purpose in solving these convex relaxations is to develop approximation results for the original NP-hard sparsest cut objective.
	
	Let $S^* \subset V$ be the set of nodes inducing the minimum sparsest cut partition of $G$, so that
	
	\[ \phi^* = \frac{\cut(S^*)}{|S^*|} + \frac{\cut(S^*)}{|\bar{S}^*|} = \frac{n \cut(S^*)}{|S^*| |\bar{S}^*|}. \]
	
	Without loss of generality, assume $|S^*| \leq |\bar{S}^*|$. In the statement of the theorem we assume that $G$ is connected, $n > 4$, and $|S^*| > 1$. The connectivity of $G$ ensures the problem can't be trivially solved by finding a single connected component, and guarantees that $\cut(S^*) \geq 1$. Together the remaining two assumptions guarantee that $\frac{ n}{|S^*| |\bar{S}^*|} \leq \frac{ n}{2(n-2)} \leq 1$, which will be useful later in the proof. We will also use the fact that $\frac{n}{|S^*| |\bar{S}^*|} \leq \frac{n\cut(S^*)}{|S^*| |\bar{S}^*|} = \phi^*$. Note that if $n \leq 4$, the problem is trivial to solve by checking all possible partitions, and if $|S^*| = 1$ then the minimum sparsest cut problem is easy to solve by checking all $n$ partitions that put a single node by itself. 
	
	In order to encode the optimal partition as a vector, define $\vs^* = (s_{ij}^*)$ by
	\[ s^*_{ij} = \begin{cases} \frac{n}{|S^*| |\bar{S}^*|} & \text{ if nodes $i$ and $j$ are on opposite side of the partition $\{S^*, \bar{S}^* \} $} \\
	0 &\text{ otherwise}. \end{cases} \]
	Observe that this vector $\vs^*$ satisfies the constraints of~\eqref{scqp} and that
	\[ \sum_{i<j} c_{ij} s^*_{ij} = \sum_{(i,j) \in E} s^*_{ij}  = \frac{\cut(S^*) n} {|S^*| |\bar{S}^*|}  = \phi^*. \]
	
	We can also prove a useful bound on the quadratic term in the objective:
	\begin{align*}
	(\vs^*)^T \mW{\vs^*} &= \sum_{i<j} w_{ij} (s_{ij}^*)^2 \\
	&= \sum_{(i,j) \in E}  (s_{ij}^*)^2 + \sum_{(i,j) \notin E} \lambda (s_{ij}^*)^2 \\
	&< \sum_{(i,j) \in E}  (s_{ij}^*)^2 + \sum_{i<j} \lambda (s_{ij}^*)^2 \\
	&= \cut(S^*) \frac{ n^2}{|S^*|^2 |\bar{S}^*|^2} + \lambda |S^*| |\bar{S}^*|\frac{n^2}{|S^*|^2 |\bar{S}^*|^2} \\
	&= \phi* \frac{ n}{|S^*| |\bar{S}^*|} + \lambda n \frac{n}{|S^*||\bar{S}^*|} \\
	&\leq \phi^*(1 + \lambda n),
	\end{align*}
	where we have used the fact that ${n}/({|S^*||\bar{S}^*|}) \leq \min \{ 1, \phi*\}$ because of our simple assumptions on $G$. With more restrictive assumptions and careful analysis we could obtain even better approximation guarantees, but our aim is simply to show for now that we can eventually obtain an $O(\log n)$ approximation for sparsest cut by minimizing a quadratic program~\eqref{scqp} instead of the original Leighton-Rao LP~\eqref{sclp}.
	
	Let $\hat{\vx}$ be the optimal solution for the QP~\eqref{scqp}, and recall that $\vs^*$ is another feasible point. We combine the bounds shown above to prove the final result:
	
	\[  \sum_{i<j} c_{ij} \hat{x}_{ij} < \sum_{i<j} c_{ij} \hat{x}_{ij} + \frac{1}{2\gamma}  \sum_{i<j} w_{ij} \hat{x}^2_{ij} \leq  \sum_{i<j} c_{ij} s^*_{ij} +  \frac{1}{2\gamma} \sum_{i<j} w_{ij} (s^*_{ij})^2 \leq \phi^* + \frac{1}{2\gamma}(1 + \lambda n) \phi^*. \]
\end{proof}

\section{Improved A Posteriori Approximations}
The approximation bounds in the previous section provide helpful suggestions for how to set parameters $\gamma$ and $\mW$ before running Dykstra's projection algorithm on a quadratic regularization of a metric-constrained LP. Once we have chosen these parameters and solved the quadratic program, we would like to see if we can improve these guarantees using the output solution for the QP.

\subsection{A First Strategy for Improved Bounds}
Consider again the optimal solutions to the LP and QP given by
\begin{align*}
\vx^* &= \argmin_\mathcal{A} \,\, \vc^T \vx \\
\hat{\vx} &= \argmin_\mathcal{A} \,\, \vc^T \vx + \frac{1}{2\gamma} \vx^T \mW \vx.
\end{align*}
where $\mathcal{A} = \{ \vx \in \mathbb{R}^N : \mA \vx \leq \vb \}$ is the set of feasible solutions. For each of the NP-hard graph clustering objectives we have considered, we have proven a sequence of inequalities of the form
\[ \vc^T \vx^* \leq  \vc^T \hat{\vx} \leq  \vc^T \hat{\vx}  + \frac{1}{2\gamma} \hat{\vx}^T \mW \hat{\vx} \leq  \vc^T \vx^* + \frac{1}{2\gamma} (\vx^*)^T \mW (\vx^*) \leq (1 + A)OPT \]
where $A$ is a term in the approximation factor (e.g. $1/\gamma$, $1/(2\gamma)$, $(1 + \lambda n)/\gamma$) and OPT is the optimal score for the NP-hard objective. If we have already computed $\hat{\vx}$, we can improve this approximation result by computing
\[ R = \frac{\hat{\vx}^T \mW \hat{\vx}}{2\gamma \vc^T \hat{\vx}}. \]
We then get an improved approximation guarantee:
\[ \vc^T \hat{\vx}  + \frac{1}{2\gamma} \hat{\vx}^T \mW \hat{\vx} = (1+R)\vc^T \hat{\vx} \implies \vc^T \hat{\vx} \leq \frac{1+A}{1+R} OPT. \]
In some cases $R$ will be small and this improvement will be minimal. However, intuitively we can see that in some special cases $R$ may be large enough to significantly improve the approximation factor. For example, it may be the case that for some correlation clustering relaxation, we choose $\gamma$ large enough so that the optimal solution to the QP, $\hat{\vm}$, and the optimal solution to the LP, $\vm^*$, are actually identical. Even after computing $\hat{\vm}$ we may not realize that $\vc^T \hat{\vx} = \vc^T \vx^*$. However, in some cases, a significant proportion of the $m^*_{ij} = \hat{m}_{ij}$ variables will close to zero or close to one. Thus, $\hat{m}_{ij} \approx \hat{m}_{ij}^2$ for many pairs $i,j$. For the correlation clustering relaxation this will mean that $\hat{\vx} \mW \hat{\vx} \approx 2 \vc^T \hat{\vx} \implies R \approx A$. Even in cases where $\vx^*$ and $\hat{\vx}$ are not identical but very close, similar reasoning shows that the above a posteriori approximation result may be much better than the a priori $(1+A)$ approximation. We note that our approximation results for correlation clustering in the experiments section are greatly aided by this a posteriori guarantee.


\subsection{Improved Guarantees by Solving a Small LP}
We outline one more approach for getting improved approximation guarantees, this time based on a careful consideration of dual variables $\hat{\vy}$ computed by Dykstra's method. This result requires a more sophisticated approach than the guarantee given in the last section. We find it extremely helpful for providing strong a posteriori guarantees when solving our quadratic relaxation of sparsest cut.

\newcommand{\hx}{\hat{\vx}}
\newcommand{\hy}{\hat{\vy}}
\newcommand{\hp}{\hat{\vp}}
\newcommand{\tx}{\tilde{\vx}}
\newcommand{\hc}{\hat{\vc}}

Once more we consider our initial linear program, which we assume is too challenging to solve using black-box software because of memory constraints:
\begin{align}
\min_{\vx} \,\, &\vc^T \vx \label{lp1}\\
\text{s.t. } &\mA \vx \leq \vb. \notag
\end{align}
We again let $\vx^*$ denote the (unknown) optimizer for~\eqref{lp1}. In practice, we solve a quadratic regularization:
\begin{align}
\min_{\vx} \,\, &\vc^T \vx + \frac{1}{2\gamma} \vx^T \mW \vx\label{qp1}\\
\text{s.t. } &\mA \vx \leq \vb. \notag
\end{align}
We solve~\eqref{qp1} by finding a primal-dual pair of vectors $(\hat{\vx}, \hat{\vy})$ satisfying KKT conditions. In particular, as noted in previous sections, these vectors satisfy
\begin{align}
\label{kkt}
&\frac{1}{\gamma} \mW \hat{\vx} = -\textbf{A}^T \hat{\vy} - \vc \\ 
\label{optqp}
&-\vb^T \hy - \frac{1}{2\gamma} \hx^T \mW \hx = \vc^T \hx + \frac{1}{2\gamma} \hx^T \mW \hx. 
\end{align}
Given this setup, we prove a new theorem for obtaining a lower bound on $\vc^T \xs$ by considering $\hy$ and solving another small, less expensive LP.

\begin{theorem}
	\label{tinylpthm}
	Given $(\hat{\vx}, \hat{\vy})$, set $\hat{\vp} = 1/\gamma \mW \hx$ and let $\tilde{\vx}$ be the optimal solution to the following new linear program:
	\begin{align}
	\max_{\vx} \,\, &\hp^T \vx \label{tinylp}\\
	\text{s.t. } & \vc^T \vx \leq \vc^T \hx \notag\\
	& \vx \in \mathcal{B} \notag
	\end{align}
	where $\mathcal{B}$ is any set which is guaranteed to contain $\vx^*$ (i.e. $\mathcal{B}$ encodes a subset of constraints that are known to be satisfied by $\vx^*$). Then we have the following lower bound on the optimal solution to~\eqref{lp1}:
	\begin{equation}
	-\vb^T \hy - \hp^T \tx \leq \vc^T \xs.
	\end{equation}
	Furthermore, if $\vx^* = \hx = \tx$, then this bound is tight.
\end{theorem}
\begin{proof}
	The dual of the original linear program~\eqref{lp1} is
	\begin{align}
	\max \,\, & -\vb^T \vy \label{duallp}\\
	\text{s.t. } & -\mA^T \vy - \vc = 0 \notag \\
	& \vy \geq 0.\notag
	\end{align}
	One way to obtain a lower bound on $\vc^T \vx^*$ would be to find some feasible point $\vy$ for~\eqref{duallp}, in which case $- \vb^T \vy \leq \vc^T \vx^*$. Note that we have access to a vector $\hy$ satisfying $\hy \geq 0$ and $\textbf{A}^T \hy - \vc = \hp = (1/\gamma)\mW \hx$. This $\hy$ is not feasible for~\eqref{duallp}, but we note that if the entries of $\hp$ are very small (which they will be for large $\gamma$), then the constraint $\textbf{A}^T \vy - \vc = 0$ is \emph{nearly} satisfied by $\hy$. If we define a new vector $\hat{\vc} = \vc + \hp$, then we can observe that $\hy$ is feasible for a slightly perturbed linear program:
	\begin{align}
	\max \,\, & -\vb^T \vy \label{pertdual}\\
	\text{s.t. } & -\mA^T \vy - \hc = 0 \notag \\
	& \vy \geq 0.\notag
	\end{align}
	We realize that this is the dual of a slight perturbation of the original LP~\eqref{lp1}:
	\begin{align}
	\min_{\vx} \,\, &\hc^T \vx \label{pertlp}\\
	\text{s.t. } &\mA \vx \leq \vb. \notag
	\end{align}
	Since $\hy$ is feasible for~\eqref{pertdual} and~$\xs$ is feasible for~\eqref{pertlp}, we have the following inequality:
	\begin{equation}
	-\vb^T \hy \leq \hc^T \vx^* = \vc^T \vx^* + \hp \vx^*. \label{laststep}
	\end{equation}
	Finally, observe that $\vx^*$ is feasible for the LP~\eqref{tinylp} defined in the statement of the theorem, and therefore $\hp \vx^* \leq \hp \tx$. Combining this fact with~\eqref{laststep} we get our final result:
	\[ -\vb^T \hy \leq \vc^T \vx^* + \hp \vx^* \leq  \vc^T \vx^* + \hp \tx \implies -\vb^T \hy - \hp^T \tx \leq \vc^T \xs. \]
	If we happen to choose $\gamma > 0$ and $\mW$ in such a way that $\vx^* = \hx$, and then pick a set $\mathcal{B}$ so that $\tx = \vx^*$, then property~\eqref{optqp} ensures that this bound will be tight. 
\end{proof}
Typically it will be difficult to choose parameters in such a way that $\vx^* = \tx = \hx$. However, the fact that this bound is tight for a certain choice of parameters is a good sign that the bound will not be too loose to be useful in practice as long as we choose parameters carefully. 

\subsection{A Bound for Sparsest Cut}
Consider the quadratic regularization of the sparsest cut relaxation shown in~\eqref{scqp}, with diagonal weight matrix defined as in Section~\ref{scbound}. Assume $(\hx, \hy)$ is the set of primal and dual variables obtained by solving the objective with Dykstra's method. We give a corollary to Theorem~\ref{tinylpthm} that shows how to obtain good a posteriori approximations for how close $\vc^T \hx$ is to the original LP relaxation of sparsest cut~\eqref{sclp}.
\begin{corollary}
	\label{scapost}
	Let $\tx = (\tilde{x}_{ij})$ be the optimizer for the following LP:
	\begin{equation}
	\label{tinyLR}
	\begin{array}{lll} \text{maximize} &(1/\gamma) \sum_{i<j} (w_{ij} \hat{x}_{ij}) x_{ij}\\ \subjectto & \sum_{i<j} x_{ij} = n & \\ & \sum_{(i,j) \in E} x_{ij} \leq \sum_{(i,j) \in E} \hat{x}_{ij} \\ &0 \leq x_{ij}\leq \frac{n}{n-1} &  \text{ for all $i,j$}. \end{array}
	\end{equation}
	Then 
	\[ n \hat{y}_1 - n\hat{y}_2 -\frac{1}{\gamma} \sum_{i<j} w_{ij} \hat{x}_{ij} \tilde{x}_{ij} \leq \sum_{(i,j) \in E} x_{ij}^* \]
	where $\hat{y}_1$ and $\hat{y}_2$ are correction variable within the dual vector $\hy$, corresponding to the constraints $\sum_{i<j} x_{ij} \leq n$ and $-\sum_{i<j} x_{ij} \leq -n$ respectively. These two constraints combine to form the equality constraint $\sum_{i<j} x_{ij} = n$.
\end{corollary}

\begin{proof}
	We just need to show that the assumptions of Theorem~\ref{tinylpthm} are satisfied. Let $x_{ij}^*$ be the optimal solution vector for the sparsest cut LP relaxation~\eqref{sclp}. Note that $x_{ij}^* \leq n/(n-1)$ for all $i,j$. If this were not the case and $x^*_{uv} > n/(n-1)$ for some pair $(u,v)$, then there would exist $(n-2)$ nodes $k$ distinct from $u$ and $v$ such that
	\[ \frac{n}{n-1} < x^*_{uv} \leq x^*_{uk} + x^*_{vk}. \]
	Then
	\[ \sum_{i<j} x^*_{ij} \geq x^*_{uv} + \sum_{u \neq k \neq v} x^*_{uk} + x^*_{vk} > \frac{n}{n-1} + (n-2)\frac{n}{n-1} = n, \]	
	which contradicts the fact that the entries of $\vx^*$ sum to $n$. We see then that all the constraints included in LP~\eqref{tinyLR} are satisfied by $x_{ij}^*$, so the result holds. 
\end{proof}

We will show in our upcoming experiments section that this bound is very helpful in guaranteeing the result of our projection method are very close to the solution of the original LP relaxation of sparsest cut. For the majority of our experiments we set parameters to $\lambda = 1/n$ and $\gamma = 5$, which guarantees a priori that the optimal QP solution will be within a factor 1.2 of the minimum sparsest cut score (Theorem~\ref{scthm}). In practice, it takes only a few seconds to solve LP~\eqref{tinyLR} after minimizing the quadratic objective, and this significantly improves the approximation guarantee. In the worst case out of 12 graphs, we use it to show that we are within 1.05 of the optimal {LP} lower bound for sparsest cut for the graph USAir97. Incidentally, we are able to compute the exact optimal LP solution for this graph using Gurobi, and we find that the actual approximation is 1.04, so our bound is very close. Typically the bound is able to confirm that our solution is within 1\% of the optimal LP score. When we decrease $\gamma$ to 2 and set $\lambda = 1/1000$ for a graph with 3086 nodes, the a posteriori approximation drops to 1.17, but this is still far better than the result we obtain by invoking the a priori guarantee in Theorem~\ref{scthm}.
\section{Experiments}
We implement Dykstra-based solvers for relaxations of sparsest cut (DykstraSC) and correlation clustering (DykstraCC)
in the Julia programming language. 
Most previous metric-constrained LP solvers have managed to obtain results only on graphs with~500 or
fewer nodes~\cite{van2007correlation,Agarwal2008metricmod,glasner2011contour}. In contrast, Dhillon et al.\ apply their triangle-fixing algorithm to solve metric nearness problems on random $n \times n$ dissimilarity matrices with~$n$ up
to 5000~\cite{dhillon2004tfa}. However, their method simply runs Dykstra's method until the change in the solution vector falls below a certain threshold. Since this approach does not take constraint satisfaction or duality gap into consideration, it comes with no output guarantees. Here we use DykstraSC to solve the sparsest cut relaxation on real-world graphs with up to 3068 nodes.
Our method is able to satisfy constraints to within machine precision, and our choice of $\gamma$ allows us to obtain strong guarantees with respect to the optimal sparsest cut score.
We also solve the correlation clustering relaxation to within a small constraint tolerance on signed, weighted graphs
with up to~11,204 nodes. This corresponds to solving a quadratic program with over $7 \times 10^{11}$ constraints. 



\subsection{Using Gurobi Software}
We compare our algorithms against black-box Gurobi optimization software. A free academic license for Gurobi software can be obtained online at \url{Gurobi.com}. When comparing our algorithm against black-box software, we take care to ensure as fair of a comparison as possible. Gurobi possesses a number of underlying solvers for LPs. In practice we separately run Gurobi's barrier method (i.e. the interior point solver), the primal simplex method, and the dual simplex method, to see which performs the best. For the interior point method, Gurobi's default setting is to convert any solution it finds to a basic feasible solution, but we turn this setting off since we do not require this of our own solver and we are simply interested in finding any solution to the LP. In practice we find that the interior point solver is the fastest. The runtimes we report do not include the time spent forming the constraint matrix. This in and of itself is an expensive task that must be taken into account when using black-box software to solve problems of this form.

\subsubsection{Lazy-Constraint Method}
Both for sparsest cut and correlation clustering we also test out an additional \emph{lazy-constraint} method when employing Gurobi software. This procedure works as follows:
\begin{enumerate}
	\item Given a metric-constrained LP, solve the objective on a subproblem that includes all the same constraints \emph{except} metric constraints.
	\item Given the solution to the subproblem, check for violations in the metric constraints. Update the constraint set to include all such violated constraints. Re-solve the LP using black-box software on the updated set of constraints.
	\item Continually re-solve the problem, check for violations, and update the constraint set. If we reach a point when all original metric constraints are satisfied before the algorithm fails due to memory issues, the solution is guaranteed to be the solution to the original metric-constrained LP.
\end{enumerate}
This procedure in some cases leads to significantly improved runtimes since it may permit us to solve the original LP without ever forming the entire $O(n^3) \times O(n^2)$ constraint matrix. Quite often we find, especially for correlation clustering problems, that many constraints will naturally be satisfied without explicitly including them in the problem setup. However, for the sparsest cut relaxation, we find that a large number of metric constraints are tight at optimality, and therefore must be included explicitly in the constraint set. In practice therefore we observe that for the sparsest cut relaxation, Gurobi continues to add constraints until a very large percentage of the original constraints are included explicitly. It therefore typically does not save time or space to repeatedly solve smaller subproblems.

\subsection{Real-world Graphs}
In our experiments we use real-world networks obtained almost exclusively from the SuiteSparse Matrix Collection~\cite{suitesparse}. In particular we use graphs from the Newman, Arenas, Pajek, and MathWorks groups for our sparsest cut experiments. In our correlation clustering experiments, we use Power, from the Newman group, and three collaboration networks from the SNAP repository~\cite{snapnets}. 

The graphs fall into the following categories:
\begin{itemize}
	\item \textbf{Citation networks}: SmallW and SmaGri
	\item \textbf{Collaboration networks}: caGrQc, caHepTh, caHepPh, Netscience, Erdos991
	\item \textbf{Power grid}: Power
	\item \textbf{US flights graph}: USAir97
	\item \textbf{Web-based graphs}: Harvard500 (web matrix), Polblogs (links between political blogs), Email (email correspondence graph)
	\item \textbf{Word graph}: Roget (thesaurus associations)
	\item \textbf{Biology networks} C. El-Neural (neural network for nematode C. Elegans), C. El-Meta (metabolic network for C. Elegans).
\end{itemize}
We also run one experiment on a graph not included in the SuiteSparse Matrix Collection. The graph Vassar85 is a snapshot of the Facebook network at Vassar College from the Facebook100 datasets. We include it in order to run our algorithm on an undirected network with around 3000 nodes.

Before running experiments on any of the graphs above, we make all edges undirected, remove edge weights, and find the largest connected component. In this way we ensure we are always working with connected, unweighted, and undirected networks.

\subsection{The Sparsest Cut Relaxation}
We run DykstraSC on ranging in size from~198 to~3068 nodes. Our machine has two 14-core 2.66 GHz Xeon processors and for ease of reproducibility we limit experiments to~100GB of RAM. 
Results are shown in Table~\ref{skyLR} and Figure~\ref{runtime}. Gurobi has an advantage on smaller graphs, but slows down and then run out of memory once the graphs scale beyond a few hundred nodes. Since DykstraSC is in fact optimizing a quadratic regularization of the sparsest cut LP relaxation, we also report how close our solution is to the optimal LP solution, either by comparing against Gurobi or using our a posteriori approximation guarantee presented in Corollary~\ref{scapost}. In nearly all cases we are within 1\% of the optimal LP solution.

When running Gurobi, for graphs with fewer than 500 nodes we have run all three solvers (interior point, dual simplex, and primal simplex). We report times for the interior point solver, since it proves to be the fastest in all cases. Gurobi runs out of memory when trying to form the entire constraint matrix for larger problems. We also test the lazy-constraint method to find it yields almost not benefit for the sparsest cut relaxation. For graphs smaller than Harvard500, where Gurobi was able to work with the entire constraint matrix, coupling the interior point solver with the lazy-constraint procedure leads to much longer runtimes. Additionally, we find in all cases that by the time the lazy-constraint solver converged, well over half of the original constraint set had to be explicitly included in order to force all other metric constraints to be satisfied. Therefore, in addition to significantly worse runtimes, we see only a minor decrease in the memory requirement.

On larger graphs, the slight decrease in memory afforded by the lazy-constraint method does allows us to solve the sparsest cut relaxation on Harvard500, which was not possible when forming the entire constraint matrix up front. This is the only positive result we see for using this approach for this relaxation. However, it still requires solving a large number of expensive subproblems, leading to a runtime that is an order of magnitude slower than DykstraSC. We also tried the lazy-constraint approach on Roget, SmaGri, Email, and Polblogs. For all of these graphs, Gurobi spends a considerable amount of time solving subproblems, but still eventually runs out of memory before finding a solution. Due to this repeated failure to produce results on much smaller graphs, we did not attempt to run the lazy-constraint solver on Vassar85.

\begin{table}[h]
	\caption{We solve the LP relaxation for sparsest cut via DykstraSC on~13 graphs.
For Vassar85, we set~$\gamma = 2$ and~$\lambda = 1/1000$; for all other datasets we set~$\gamma = 5$ and~$\lambda = 1/n$. Both DykstraSC and Gurobi (when it doesn't run out of memory) solve the problems to within a relative gap tolerance of~$10^{-4}$, and satisfy constraints to within machine precision. The last column reports an upper bound on the ratio between the LP score produced by DykstraSC and the optimal LP solution. Time is given in seconds.}
	\centering
	\begin{tabular}{lrrrrrrr}
		\toprule
		Graph & $|V|$ & $|E|$  & \# constraints & Gurobi Time & Dykstra Time & Approx \\
		\midrule
		Jazz & 198 & 2742 & $3.8 \times 10^6$& 60 &  81 & 1.003 \\ 
		SmallW & 233 & 994 & $6.2 \times 10^6$& 93 & 166& 1.001 \\ 
		C.El-Neural & 297 & 2148 &  $1.2 \times 10^{7}$ & 274 & 350 & 1.000 \\ 
		USAir97 & 332 & 2126 & $1.8 \times 10^7$ & 471 & 511 & 1.041 \\ 
		Netscience & 379 & 914 & $2.7 \times 10^{7}$ & 887 & 1134 & 1.000 \\ 
		Erdos991 & 446 & 1413 & $4.4 \times 10^{7}$ & 2574  & 1954 & 1.011 \\ 
		C.El-Meta & 453 & 2025 & $4.6 \times 10^{7}$ & 2497 & 1138 & 1.000\\ 
		Harvard500 & 500 & 2043 & $6.2 \times 10^{7}$ & 18769  & 1427 & 1.000 \\ 
		Roget & 994 & 3640 & $4.9 \times 10^{8}$ & out of memory & 53449 & 1.008 \\
		SmaGri & 1024 & 4916 & $5.4 \times 10^{8}$ & out of memory  & 25703 & 1.002 \\ 
		Email & 1133 & 5451 & $7.3 \times 10^{8}$& out of memory  & 34621 & 1.005 \\ 
		Polblogs & 1222 & 16714 & $9.1 \times 10^{8}$ & out of memory  & 41080 & 1.013 \\ 
		Vassar85 & 3068 & 119161 & $1.4 \times 10^{10}$ & out of memory  & 155333& 1.165 \\ 	
		\bottomrule 
	\end{tabular}
	\label{skyLR}
\end{table}

\begin{figure}
	\begin{minipage}[c]{0.4\linewidth}
		\centering
	\includegraphics[width=\linewidth]{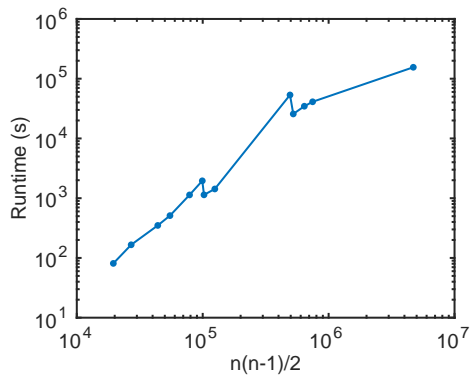}
	\end{minipage}
	\hspace{1cm}
	\begin{minipage}[c]{0.45\linewidth}
		\caption{Runtimes for DykstraSC on real-world graphs with 198 to 3068 nodes. If $n$ is the number of nodes in the graph, then DykstraSC solves for $n(n-1)/2$ distance scores.}
			\label{runtime}
	\end{minipage}
\end{figure}

\subsection{Weighted Correlation Clustering}
We convert several real-world graphs into instances of correlation clustering using the approach of Wang et
al.~\cite{wang2013scalable}. The procedure is as follows:
\begin{enumerate}
	\item Given an input graph $G = (V,E)$, compute the Jaccard coefficient between each pair of nodes $i,j$:
	\[ J_{ij} = \frac{|N(i)\cap N(j)|}{|N(i)\cup N(j)|} \]
	where $N(u)$ is the set of nodes adjacent to node $u$.
	\item Apply a non-linear function on Jaccard coefficients to obtain a score indicating similarity or dissimilarity:
	\[ S_{ij} = \log\left(  \frac{1 + (J_{ij}- \delta)}{1-(J_{ij} - \delta)}\right). \]
	Here, $\delta$ is a parameter set so that $S_{ij} > 0$ if $J_{ij} > \delta$ and $S_{ij} < 0$ when $J_{ij} < \delta$. Following Wang et al.~\cite{wang2013scalable}, we fix $\delta = 0.05$.
	
	\item Wang et al.\ stop after the above step and use $S_{ij}$ scores for their correlation clustering problems. We additionally offset each entry by $\pm \epsilon$ to avoid cases where edge weights are zero:
	\[ Z_{ij} = \begin{cases} S_{ij} + \epsilon &\text{ if $S_{ij} > 0$} \\  
	S_{ij} - \epsilon & \text{if $S_{ij} < 0$} \\ 
	\epsilon & \text{ if $S_{ij} = 0$ and $(i,j) \in E$} \\ 
	-\epsilon & \text{ if $S_{ij} = 0$ and $(i,j) \notin E$}. \\
	\end{cases} \]
	If $S_{ij} = 0$, this indicates there is no strong similarity or dissimilarity between nodes based on their Jaccard coefficient. If in this case nodes $i$ and $j$ are adjacent, we interpret this as a small indication of similarity and assign them a small positive weight. Otherwise we assign a small negative weight. In all our experiments we fix $\epsilon = 0.01$.
\end{enumerate}
The sign of $Z_{ij}$ indicates whether nodes $i$ and $j$ are similar or dissimilar, and $w_{ij}  = |Z_{ij}| > 0$ is the non-negative weight for the associated correlation clustering problem.
Results for running DykstraCC and Gurobi on the resulting signed graphs are shown in Table~\ref{skyCC}.

On problems of this size, we have no hope of ever forming the entire constraint matrix and using Gurobi without running out of memory. Therefore, we restrict to using the lazy-constraint approach, coupled with Gurobi's interior point solver. In one case, the lazy-constraint method converges very quickly. Effectively, it finds a small subset of constraints that are sufficient to force
all other metric constraints to be satisfied at optimality. However, Gurobi runs out of memory on the other large problems considered, indicating that, even if we are extremely careful, black-box solvers are unable to compete with our Dykstra-based approach.

Because the correlation clustering problems we address are so large, we set $\gamma = 1$ and run Dykstra's method until constraints are satisfied to within a tolerance of 0.01. We find that long before the constraint tolerance reaches this point, the duality gap shrinks below $10^{-4}$. We note that although it takes a long time to reach convergence on graphs with thousands of nodes, DykstraCC has no issues with memory. Monitoring the memory usage of our machine, we noted that for the 11,204 node graph, DykstraCC was using only around~12 of the 100GB of available RAM. Given enough time therefore, we expect our method to be able to solve metric-constrained LPs on an even much larger scale. The ability to solve these relaxations on problems of this scale is already an accomplishment, given the fact that standard optimization software often fails on graphs with even a few hundred nodes. In future work we wish to continue exploring options for randomized and parallel projection methods, so that we can more quickly obtain answers for large-scale correlation clustering relaxations and other related problems.

\begin{table}[t]
	\caption{DykstraCC can solve convex relaxations of correlation clustering with up to~700 billion
constraints. The lazy-constraint Gurobi solver does very well for one very sparse graph, but runs out of memory on all
other problems. We set~$\gamma = 1$, and constraint tolerance to~$0.01$. 
Selecting a small~$\gamma$ leads to poorer approximation guarantees, but dramatically decreases the number of needed iterations until convergence. We can still obtain the standard~$O(\log n)$ approximation for weighted correlation clustering by applying previous rounding techniques~\cite{Demaine2003,Emanuel2003,charikar2005clustering}. 
}
	\centering
	\begin{tabular}{lrrrrrrr}
		\toprule
		Graph & $|V|$ & $|E|$  & \# constraints & Gurobi Time & Dykstra Time & Approx \\
		\midrule
				power & 4941 & 6594  & $6.0 \times 10^{10}$& 549 s & 7.6 hrs & 1.07 \\  
				caGrQc & 4158 & 13422 & $3.6 \times 10^{10}$ & out of memory & 6.6 hrs & 1.33 \\ 
				caHepTh & 8638 & 24806 & $3.2 \times 10^{11}$ & out of memory  & 88.3 hrs & 1.34 \\ 
				caHepPh & 11204 & 117619 & $7.0 \times 10^{11}$ & out of memory  & 167.5 hrs & 1.27 \\
		\bottomrule 
	\end{tabular}
	\label{skyCC}
\end{table}

\subsection{Experimental Differences Between Metric LPs}
We note that the correlation clustering relaxation appears categorically easier to solve than the sparsest cut relaxation in our experiments. Dykstra's method tends to generate a denser dual vector for the sparsest cut relaxation, leading to memory issues on relatively small problems. The lazy-constraint method, which can successfully solve some correlation clustering problems with thousands of nodes, appears to provide no benefit for the sparsest cut relaxation. These observations are perhaps surprising, given that the underlying constraint set for both problems is nearly identical. 

In order to explain this phenomenon, we first consider the original clustering objectives, rather than their LP relaxations. Note that the minimum sparsest cut always partitions a graph into two clusters, which often are at least somewhat balanced in size. Consider what this means for the binary variables $x_{ij}$ that encode the clustering: in the case of two balanced clusters, many of these variables will be zero. In other words, there will be a large number of tight triangle inequality constraints of the form $x_{ij} = 0 = 0 + 0 = x_{ik} + x_{jk}$, because of triplets $(i,j,k)$ that are all in the same cluster. Similarly, for two nodes $i,j$ in one cluster and a third node $k$ in the other cluster, the triangle constraints will also be tight:
\[ x_{ik} = c = c + 0 = x_{jk} + x_{ij}, \]
where $c >0$ is some constant depending on the graph and the underlying partition (see Section~\ref{scbound} for details). On the other hand, the optimal correlation clustering problem will often partition a graph into a large number of clusters. In this case, triplets of nodes $(i,j,k)$ that are all in distinct clusters are much more prevalent, and such triplets correspond to metric constraints that are not tight at optimality ($x_{ij} = c < c+ c = x_{ik} + x_{jk}$). When we relax these clustering objectives so that distances $x_{ij}$ are no longer binary, we expect this phenomenon to still be reflected to some degree or another in the relaxed distances. 

Intuitively, given that the sparsest cut relaxation typically involves a larger number of tight constraints, we would expect the dual vector in our projection algorithm to become dense more quickly, since adjustments need to be made more frequently and carefully at tight constraints. This explains the higher memory requirement we see for solving the sparsest cut relaxation. This also explains why the lazy-constraint black-box method is unsuccessful. The existence of many tight constraints implies that we need to eventually include a significant fraction of the metric constraints explicitly if we wish to solve the original problem. In practice therefore it is often best to simply include the full constraint set up front rather than solving a large number of subproblems that don't solve optimal solution.

\section{Discussion and Future Challenges}
In this paper we have developed an approach for solving expensive convex relaxations of clustering objectives that works
on a much larger scale than was previously possible. We now observe several challenges that seem inherent in improving
this approach, without significantly departing from the application of projection methods.
We tried variants of Bauschke's method~\cite{bauschke1996approx} and Haugazeau's projection
method~\cite{haugazeau1968inequations,bauschke2017convex}, which do not compute dual variables as Dykstra's does.
Such methods hence require only~$O(n^2)$ memory, instead of~$O(n^3)$, but come with significantly worse convergence guarantees. Because it is hard to determine in practice when these methods have converged, it is difficult to use them to obtain an output satisfying any guarantees with respect to the optimal solution. Another natural approach to consider is the use of parallelization or randomization. Parallel versions of
Dykstra's method exist~\cite{iusem1991convergence}, but they rely on averaging out a large number of very tiny changes
in each iteration, equal to the number of constraints. Since in our case there are~$O(n^3)$ constraints, we find that,
in practice, this averaging approach leads to changes that are so small that no meaningful progress can be made from one iteration to the next. The challenge in using a randomized approach (see~\cite{hildreths_rand}) is that visiting constraints at random leads to a much higher cost for visiting the same number of constraints. This is because accessing dual variables at random from a dictionary is slower in practice than sequentially visiting elements in an array of dual variables.

\bibliographystyle{plain}
\bibliography{triconstraints}  	

\begin{appendix}

\section{Lemma Supporting the Proof of Theorem 1}
\label{app1}
The full proof of Theorem~\ref{thm:metric-cc} relies the following lemma showing that we can safely discard upper and lower bound on $x_{ij}$ variables without changing the optimal result.

\begin{lemma}
	\label{lemma1}
	Even without the constraint family ``$0 \leq x_{ij} \leq 1$ for all pairs $(i,j)$'',
	every optimal solution to problem~\eqref{eq:cc3} in fact satisfies those constraints.
\end{lemma}

\begin{proof}
	The proof proceeds by contradiction: assume that $\mX = (x_{ij})$ is an optimal solution obtained by optimizing
	objective~\eqref{eq:cc3}, without constraints $x_{ij} \in [0,1]$.
	Assume also that at least one variable is greater than one or less than zero. Next, define a new set of variables $\mZ = (z_{ij})$ as follows:
	\[ z_{ij} = \begin{cases} 
	0 & \text{ if $x_{ij} < 0$ } \\
	x_{ij} & \text{ if $x_{ij} \in [0,1] $}\\
	1 & \text{ if $x_{ij} >1 $}\\
	\end{cases}
	\]
	Notice that this $\mZ$ would have a strictly lower (i.e. better) objective score than~$\mX$ for problem~\eqref{eq:cc3}, because $|z_{ij} - d_{ij}| \leq |x_{ij} - d_{ij}|$ for all $(i,j)$, and this inequality is strict for at least one pair of nodes since we assumed that there is at least one variable $x_{ij}$ that is not in $[0,1]$. It just remains to show that $\mZ$ also satisfies triangle inequality constraints and is thus feasible, which would lead to the desired contradiction to the optimality of $\mX$.
	
	Proving $\mZ$ is feasible is tedious, as it requires checking a long list of cases. We consider a triplet $(i,j,k)$, and we wish to check that
	\begin{align*}
	z_{ij} \leq z_{jk} + z_{ik} \\
	z_{ik} \leq z_{jk} + z_{ij} \\
	z_{jk} \leq z_{ij} + z_{ik}
	\end{align*}
	for all possible values of $(x_{ij}, x_{ik}, x_{jk})$. To illustrate the technique we will provide proofs for
	the cases where the~$\mX$ variables are all nonnegative, but may be larger than~$1$. The same approach works if we also
	considered the case where some of the $\mX$ are negative, though this involves checking~27 cases, which we do not list
	exhaustively here. Let $v_{ab}$ be a binary variable indicating whether $x_{ab} > 1$.
	\begin{itemize}
		\item Case 1: $(v_{ij}, v_{ik}, v_{jk}) = (0, 0, 0)$. 
		In this case the~$y$ variables are identical to the~$x$ variables and the constraints hold.
		\item Case 2: $(v_{ij}, v_{ik}, v_{jk}) = (1, 1, 1)$. 
		In this case $z_{ij} = z_{ik} = z_{jk} = 1$ and the constraints hold.
		\item Case 3: $(v_{ij}, v_{ik}, v_{jk}) = (1, 0, 0)$. 
		Since $x_{ij} > 1$, by construction $z_{ij} = 1$, and $z_{jk} = x_{jk} \leq 1$, $z_{ik} = x_{ik} \leq 1$, so we can confirm that:
		\begin{align*}
		z_{ij} &< x_{ij} \leq x_{ik} + x_{jk} = z_{ik} + z_{jk} \\
		z_{ik} &= x_{ik} \leq 1 \leq x_{jk} + 1 = z_{jk} +  z_{ij}\\
		z_{jk} &= x_{jk} \leq 1 \leq x_{ik} + 1 = z_{ik} + z_{ij} 
		\end{align*}
		\item Case 4: $(v_{ij}, v_{ik}, v_{jk}) = (0, 1, 0)$. 
		This is symmetric to case~3.
		\item Case 5: $(v_{ij}, v_{ik}, v_{jk}) = (1, 1, 0)$. In this case we see that $z_{ij} = 1 < x_{ij}$, $z_{ik} = 1 < x_{ik}$, and $z_{jk} = x_{jk}$, so
		\begin{align*}
		z_{ij} &= 1 < 1 + z_{jk} = z_{ik} + z_{jk} \\
		z_{ik} &= 1 < 1 + z_{jk} = z_{ij} + z_{jk} \\
		z_{jk} &< 1 \leq 1 + 1 = z_{ij} + z_{ik}.
		\end{align*}
		\item Case 6: $(v_{ij}, v_{ik}, v_{jk}) = (0, 0, 1)$. 
		Symmetric to cases 3 and 4.
		\item Case 7: $(v_{ij}, v_{ik}, v_{jk}) = (1, 0, 1)$. 
		Symmetric to case 5.
		\item Case 8: $(v_{ij}, v_{ik}, v_{jk}) = (0, 1, 1)$. 
		Symmetric to cases 5 and 7.
	\end{itemize}
\end{proof}

\section{Projection Methods and Quadratic Programming}
\label{app2}
In this appendix we provide more general background on Dykstra's projection method~\cite{dykstra1983algorithm} and how it
can be used to minimize a quadratic program. The results here are not new, but are rather a compilation of past work on
projection methods and quadratic programming which are most relevant to us.

\subsection{Quadratic Programming}
Let $\mA  \in \mathbb{R}^{N\times M}$, $\vb \in \mathbb{R}^M$, $\vc \in \mathbb{R}^N, \gamma > 0$, and $\mW \in \mathbb{R}^{N \times N}$ be a positive definite matrix. Consider the following quadratic program:
\begin{align}
\min_{\vx} \,\, &Q(\vx) = \frac{1}{2} \vx^T \mW \vx + \vc^T \vx \label{primal}\\
\text{s.t. } &\mA \vx \leq \vb. \notag
\end{align}
The dual of this quadratic program is
\begin{align}
\max_{\vy} \,\, D(\vy) = -\vb^T \vy - &\frac{1}{2} \left \|{\textbf{A}^T \vy + \vc } \right \|_{\mW^{-1}}^2 \label{dual}\\
\text{s.t. } & \vy \geq 0. \notag
\end{align}
The objectives are equal when the Karush–Kuhn–Tucker (KKT) conditions are satisfied, i.e., when we find vectors $\hat{\vx}$ and $\hat{\vy}$ that satisfy:
\paragraph{KKT Conditions}
\begin{enumerate}
	\item $\mW \hat{\vx} = -\textbf{A}^T \hat{\vy} -\vc$
	\item $\mA  \hat{\vx} \leq \vb$ 
	\item $\hat{\vy}^T (\mA  \hat{\vx}-\vb) = 0$
	\item $\hat{\vy} \geq 0$.
\end{enumerate}
If the above are satisfied, we know that $\hat{\vy}  = \argmin D(\vy)$ and $\hat{\vx} = \argmin Q(\vx)$ and $D(\hat{\vy} ) = Q(\hat{\vx})$. In general, if~$\vx$ satisfies $\mA \vx \leq \vb$ and $\vy \geq 0$ then we know that $D(\vy) \leq D(\hat{\vy} ) = Q(\hat{\vx}) \leq Q(\vx)$. 

\subsection{The Best Approximation Problem}
Let $C_i \subset \mathbb{R}^N$ for $i = 1, 2, \hdots , M$ be convex sets and $C = \cap_{i = 1}^M C_i$ be their intersection (which is also convex). Given $\vz \in \mathbb{R}^N$, the best approximation problem (BAP) is to find
\begin{equation}
\vx^* = P_C(\vz) = \arg \min_{\vx \in C}  || \vx - \vz ||^2 \label{bap}
\end{equation}
where we can specify any norm $||\cdot||$. $P_C(\vz)$ is the \emph{projection of $\vz$ onto $C$}.
BAP is often solved using \emph{projection methods}, which operate by visiting the constraints~$(C_i)$ cyclically and repeatedly performing easier projections of the form
\[ \vx_i = P_{C_i}(\vz_i) = \arg \min_{\vx \in C_i}  || \vx - \vz_i ||^2 \]
Note that we are \emph{not} restricted to only the standard Euclidean norm, in fact we will use a weighted norm in our work.

\subsection{Dykstra's Method for BAP}
Dykstra's method is one common approach to solving the BAP (see~\cite{dykstra1983algorithm}). In order to solve~\eqref{bap}, Dykstra's method computes the following updates:
\begin{enumerate}
	\item $I_i = 0 \in \mathbb{R}^N$ for $i = 1,2, \hdots, M$ (\emph{increment} or \emph{correction} vectors)
	\item $\vx := \vz$
	\item At step $k$
	\begin{itemize}
		\item $i := (k-1) \bmod M + 1$ (cyclically visit constraints)
		\item $\vx_c := \vx - I_i$ 		(correction step)
		\item $\vx := P_{C_i}(\vx_c) $ (projection step)
		\item $I_i := \vx - \vx_c$ (update increment).
	\end{itemize}
\end{enumerate}
This simplifies when we consider applying Dykstra's method to a quadratic program like~\eqref{primal}. To do this,
we will not use the standard norm and standard inner product, but, given arbitrary vectors $\vf, \vg \in \mathbb{R}^N$,
a weighted norm:
\[ \langle \vf, \vg \rangle_w = \vf^T \mW \vg \]
\[ \| \vf \|_{w}^2 = \langle \vf, \vf \rangle_w = \vf^T \mW \vf. \]
%
Observe the the following equivalence:
\newcommand{\wnorm}[1]{\ensuremath{ \left \| #1 \right \|_{w} }}
\begin{align*} 
\frac{1}{2} \wnorm{\vx + {\mW^{-1} \vc} }^2 &=  \frac{1}{2} \vx^T \mW \vx +  \vx^T \mW \mW^{-1} \vc  + \frac{1}{2} \vc^T \mW^{-1} \vc \\
&= \frac{1}{2} \vx^T \mW \vx +  \vc^T \vx  + \emph{constant}.
\end{align*}
In other words, the quadratic program~\eqref{primal} is equivalent to the following best approximation problem:
\begin{equation}
\vx^* = P_C(\vz) = \arg \min_{\vx \in C}  || \vx - \vz ||^2_{w}, \text{ where } C = \cap_{i = 1}^M C_i\,. \label{qpbap}
\end{equation}
Here, $\vz = -\mW^{-1} \vc$ and we are projecting onto half-space constraints:
\newcommand{\rn}{\mathbb{R}^N}
\[C_i = \{ \vx \in \mathbb{R}^N : \va_i^T \vx \leq b_i \} = \{\vx \in \rn : \langle \tilde{\va}_i, \vx \rangle_w \leq b_i \}\,, \]
where~$\va_i^T$ is the $i$th row of constraint matrix~$\mA$ and $\tilde{\va}_i = \mW^{-1} \va_i$ is scaled so that
\[ \langle \tilde{\va}_i, \vx \rangle_w = \tilde{\va_i}^T \mW \vx = \va_i^T \mW^{-1} \mW \vx = \va_i^T \vx. \]
Dykstra's method relies on being able to project quickly onto each constraint~$C_i$. For simple half-space constraints, the projection can be computed as follows:
\begin{equation*}
P_{C_i}(\vx) =\arg \min_{\vx' \in C_i}  || \vx' - \vx ||^2 = \vx - \frac{ [\langle \tilde{\va}_i, \vx \rangle - b_i]^+}{\|\tilde{\va}_i \|^2 } \tilde{\va}_i
\label{proj}
\end{equation*}
where $[a]^+$ is~$a$ if $a > 0$, and is zero otherwise (a standard textbook result, see section~4.1.3 of~\cite{cegielski2012iterative}).
Since we are using the $\mW$-weighted norm and inner product, for our problem this becomes:
\begin{align*}
P_{C_i}(\vx) &= \vx - \frac{ [\va_i^T \vx - b_i]^+}{\va_i^T \mW^{-1} \va_i } (\mW^{-1} \va_i).
\end{align*}
Observe that this type of projection always takes the original vector~$\vx$ and just adds a constant times a vector $\mW^{-1} \va_i$ when visiting constraint~$i$. Therefore, when implementing the algorithm, we do not need to store an entire increment vector $I_i$ for each constraint. As long as we have the weight matrix $\mW$ and the constraint matrix $\textbf{A}$ stored up front, it will suffice to store a single extra constant ${ [\va_i^T \vx - b_i]^+}/{\va_i^T \mW^{-1} \va_i }$ in order to perform the correction step.

We re-write Dykstra's algorithm for quadratic programming as follows:

\noindent \emph{Dykstra's Algorithm applied to Quadratic Program~\eqref{primal}}
\begin{enumerate}
	\item $\vy := 0 \in \mathbb{R}^M$
	\item $\vx := -\mW^{-1}\vc$
	\item At step $k$:
	\begin{itemize}
		\item $i := (k-1) \bmod M + 1$ (cyclically visit constraints)
		\item $\vx := \vx + y_i\mW^{-1} \va_i$ (correction step)
		\item $\vx := \vx - \theta_i^+ \mW^{-1} \va_i$ (projection step)
		
		where $\theta_i^+ = \frac{ [\va_i^T \vx - b_i]^+}{\va_i^T \mW^{-1} \va_i }$
		\item $y_i := \theta_i^+ \geq 0.$
	\end{itemize}
\end{enumerate}
\subsection{Hildreth's Projection Method}
It is well known that for half-space constraints, Dykstra's method is equivalent to Hildreth's method~\cite{hildreth1957quadratic}. At first glance there appear to be slight differences, but these are easily accounted for. 

\noindent \emph{Hildreth's Algorithm}
\begin{enumerate}
	\item $\vy := 0 \in \mathbb{R}^M$
	\item $\vx := -\mW^{-1}\vc$
	\item At step $k$:
	\begin{itemize}
		\item $i := (k-1) \bmod M + 1$
		\item $\theta_i = \frac{ \va_i^T \vx - b_i}{\va_i^T \mW^{-1} \va_i }$
		\item $\delta = \min\{-\theta_i, y_i \}$
		\item $\vx := \vx + \delta \mW^{-1} \va_i$
		\item $y_i := y_i - \delta.$
	\end{itemize}
\end{enumerate}
Proving the equivalence between methods amounts simply to combining the two separate updates:
\begin{itemize}
	\item $\vx_c = \vx + y_i\mW^{-1} \va_i$ (correction step)
	\item $\vx' := \vx_c -  \frac{ [\va_i^T \vx_c - b_i]^+}{\va_i^T \mW^{-1} \va_i }\mW^{-1} \va_i$ (projection step)
\end{itemize}
from Dykstra's method. Combining these updates gives:
\[ \vx' = \vx + y_i\mW^{-1} \va_i -  \frac{ [\va_i^T (\vx + y_i\mW^{-1} \va_i) - b_i]^+}{\va_i^T \mW^{-1} \va_i }\mW^{-1} \va_i.\]
The outcome of this update depends on the value of 
\[ \va_i^T (\vx + y_i\mW^{-1} \va_i) - b_i = (\va_i^T\vx + y_i \va_i^T \mW^{-1} \va_i) - b_i  \]
which is greater than or equal to zero if and only if:
\[ -\theta_i = -\frac{\va_i^T \vx - b_i}{\va_i^T \mW^{-1} \va_i} \leq y_i.\]
Therefore, if $-\theta_i \leq y_i$, then $\delta = -\theta_i$ and the update is:
\begin{align*}
\vx' &= \vx + y_i\mW^{-1} \va_i -  \frac{ \va_i^T\vx - b_i + y_i\va_i^T\mW^{-1} \va_i}{\va_i^T \mW^{-1} \va_i }\mW^{-1} \va_i\\
&= \vx + y_i\mW^{-1} \va_i -  (\theta_i + y_i) \mW^{-1}\va_i\\
&= \vx - \theta_i \mW^{-1}\va_i\\
&= \vx + \delta \mW^{-1}\va_i.
\end{align*}
On the other hand, if $y_i < -\theta_i$, then $[\va_i^T (\vx + y_i\mW^{-1} \va_i) - b_i]^+ = 0$ and the update is $\vx' = \vx + y_i\mW^{-1} \va_i$.

\subsubsection{Nonnegative Dual Variables.} Note that $y_i \geq 0$ is maintained in either case: if $\delta = y_i$ then we update $y_i := y_i - \delta = 0$, and if $\delta = -\theta_i$ that means $-\theta_i \leq y_i \implies 0 \leq y_i + \theta_i = y_i - \delta$. This is important, because as we will show in the next section, these $y_i$ variables are the nonnegative dual variables in the optimization problem.

\subsection{Hildreth's Method, Coordinate Descent, and KKT Conditions}
Now we will show the equivalence between Dykstra's/Hildreth's method and performing coordinate descent on the negative of the dual function. The details shown here are a slight generalization of the result shown by Dax~\cite{daxtheory,simplealg}, updated to explicitly include a non-identity weight matrix $\mW$.

We first replace the maximization objective~\eqref{dual} with an equivalent quadratic program that is minimized:
\begin{align}
\min_{\vy} \,\, F(\vy) = \vb^T \vy + &\frac{1}{2} \left \|{\textbf{A}^T \vy + \vc } \right \|_{\mW^{-1}}^2 \label{fobj}\\
\text{s.t. } & \vy \geq 0 \notag
\end{align}
and consider the following optimization approach: let $\vy = 0$ and set $\vx = -\mW^{-1} \vc$ so that the first KKT condition is satisfied: $\mW\vx = -\mA^T \vy - \vc$. We will update $\vy$ one variable at a time so that $F(\vy)$ is strictly decreasing and so that the first and fourth KKT conditions are satisfied at every step: $\mW\vx_k = -\mA^T \vy_k - \vc$ and $\vy_k \geq 0$.

\subsubsection{Maintaining Two KKT Conditions}
When we visit the $i$th constraint, we perform the following steps:
\begin{enumerate}
	\item Compute $\theta_i = \frac{ \va_i^T \vx - b_i}{\va_i^T \mW^{-1} \va_i }$
	\item Set $\delta = \min\{-\theta_i, y_i \}$
	\item Update $\vx := \vx + \delta \mW^{-1} \va_i$, and $\vy := \vy - \ve_i\delta,$
\end{enumerate}
where $\ve_i$ is the vector with zeros everywhere except for a 1 in the $i$th position. We have already noted that the entries of $\vy$ computed by the method will always be nonnegative. To see that one other KKT condition is also maintained, assume that $\mW \vx = -\mA^T \vy - \vc$ holds at one point in time and perform a single update to get new primal and dual vectors $\vx'$ and $\vy'$:
\begin{align*}
\vx' = \vx + \delta \mW^{-1} \va_i \\
\vy' = \vy - \ve_i\delta.
\end{align*}
Then note:
\[ \vx' = \vx + \delta \mW^{-1} \va_i = \mW^{-1}(-\mA^T \vy - \vc)  + \delta \mW^{-1} \va_i  = \mW^{-1}(-\mA^T \vy + \delta\va_i- \vc)\]
and
\[ -\mA^T \vy' - \vc = -\mA^T (\vy - \ve_i\delta) - \vc = -\mA^T \vy - \vc + \delta \va_i \]
so these combine to yield $\mW \vx' = -\mA^T \vy' - \vc $.

\subsubsection{Coordinate Descent}
The connection to coordinate descent is seen by realizing that the value $\theta_i$ computed above uniquely minimizes the following one-variable function:
\begin{align*}
f(\theta)& = F(\vy + \ve_i \theta)\\
& = \vb^T \vy + \theta b_i + \frac{1}{2} \left \| \mA^T (\vy + \ve_i \theta ) +\vc \right \|_{\mW^{-1}}^2 \\
& = \vb^T \vy + \theta b_i + \frac{1}{2} \left \| -\mW \vx + \theta \va_i  \right \|_{\mW^{-1}}^2\\
& = \vb^T \vy + \theta b_i + \frac{1}{2} \left( \left( -\mW \vx + \theta \va_i  \right )^T \mW^{-1} \left( -\mW \vx + \theta \va_i  \right ) \right) \\
& = \vb^T \vy + \theta b_i + \frac{1}{2} \left( \vx^T \mW \vx + \theta^2 \va_i^T \mW^{-1} \va_i - 2\theta \va_i^T \vx \right) \\
& = \theta b_i + \frac{1}{2} \theta^2 \va_i^T \mW^{-1} \va_i - \theta \va_i^T \vx  + \frac12 \vx^T \mW \vx + \vb^T \vy.
\end{align*}
This is minimized when 
\[f'(\theta) = b_i - \va_i^T \vx + \theta \va_i^T \mW^{-1} \va_i = 0 \implies \theta = \frac{ \va_i^T \vx - b_i}{\va_i^T \mW^{-1} \va_i }.\]
If we can perform this update without violating constraint $y_i \geq 0$, then we do so, i.e. $\delta = \theta_i$ and update $y_i = y_i + \theta \geq 0$. Otherwise if $\theta < -y_i \leq 0$, we simply decrease $y_i$ as much as we can without violating the constraint (i.e. set $y_i = 0$). The function strictly decreases, since
\[f(0) - f(\delta) \geq \frac{1}{2} \delta^2 \va_i^T \mW^{-1} \va_i >0. \]
To see this, realize that $\delta = \nu \theta_i$ for some $\nu \in [0,1]$ and $\theta_i (\va_i^T \mW^{-1} \va_i) = (\va_i^T \vx - b_i)$, so
\begin{align*}
f(0) - f(\delta) &= \delta (\va_i^T \vx - b_i) - \delta^2/2 \va_i^T \mW^{-1} \va_i \\
&= \nu \theta_i (\theta_i \va_i^T \mW^{-1} \va_i ) - \frac12 \nu^2 \theta_i^2 \va_i^T \mW^{-1} \va_i \\
&= \frac12 \theta_i^2\nu^2 \va_i^T \mW^{-1} \va_i (2 - \nu)/\nu \\
&\geq \frac12 \delta^2 \va_i^T \mW^{-1} \va_i \hspace{.5cm} \text{ (since $\nu \in [0,1]$). }
\end{align*}

\end{appendix}

\end{document}